\documentclass[12pt]{amsart}
\usepackage{amsmath,amsthm,amsfonts,amssymb,a4wide,url,mathscinet,mathdots}
\usepackage[all]{xypic}

\newcommand{\A}{\mathbb{A}}
\newcommand{\Z}{\mathbb{Z}}
\newcommand{\Q}{\mathbb{Q}}
\newcommand{\C}{\mathbb{C}}

\newcommand{\tr}{\operatorname{tr}}
\renewcommand{\Re}{\operatorname{Re}}
\newcommand{\simw}{\sim_w}
\newcommand{\fsim}{\sim'}
\newcommand{\der}{\operatorname{der}}
\newcommand{\univ}{\mathcal{U}}
\newcommand{\aaa}{\mathfrak{a}}
\newcommand{\nnn}{\mathfrak{n}}
\newcommand{\hhh}{\mathfrak{h}}
\newcommand{\Lieg}{\mathfrak{g}}

\newcommand{\GL}{\operatorname{GL}}
\newcommand{\prmlvl}{L}
\newcommand{\iii}{\textrm{i}}
\newcommand{\FFF}{\mathcal{F}}

\newcommand{\modulus}{\delta}
\newcommand{\geom}[1]{\mathbf{#1}}
\newcommand{\fin}{\operatorname{fin}}
\newcommand{\sprod}[2]{\left\langle#1,#2\right\rangle}  
\newcommand{\bs}{\backslash}
\newcommand{\conj}{\mathbf{C}}
\newcommand{\affine}{\mathbb{G}_a}
\newcommand{\spart}[1]{#1_s}
\newcommand{\upart}[1]{#1_u}

\newcommand{\smofrts}[2]{\xi_{#1}^{#2}}
\newcommand{\asmofrts}[2]{\tilde\xi_{#1}^{#2}}
\newcommand{\sctn}{\mathfrak{s}}
\newcommand{\ngamma}{\kappa}
\newcommand{\quotmap}{q}
\newcommand{\ad}{\operatorname{ad}}
\newcommand{\vol}{\operatorname{vol}}
\newcommand{\siegel}{\mathfrak{S}}
\newcommand{\abs}[1]{\left|{#1}\right|}
\newcommand{\norm}[1]{\lVert#1\rVert}
\newcommand{\coor}{\mathfrak{c}}
\newcommand{\Ad}{\operatorname{Ad}}
\newcommand{\pars}{\mathfrak{P}}
\newcommand{\Diff}{\mathcal{D}}
\newcommand{\fund}{\mathcal{B}}
\newcommand{\Ht}{H}
\newcommand{\roots}{\Sigma}
\newcommand{\R}{\mathbb{R}}
\newcommand{\funct}{\mathcal{C}}

\newcommand{\K}{\mathbf{K}}
\newcommand{\Orb}{\mathfrak{O}}
\newcommand{\orb}{\mathfrak{o}}
\newcommand{\cOrb}{{\tilde{\mathfrak{O}}}}
\newcommand{\corb}{{\tilde{\mathfrak{o}}}}
\newcommand{\rg}[1]{_{#1}^\circ}

\newcommand{\level}{\operatorname{level}}

\newtheorem{theorem}{Theorem}[section]
\newtheorem{lemma}[theorem]{Lemma}

\newtheorem{proposition}[theorem]{Proposition}
\newtheorem{corollary}[theorem]{Corollary}
\theoremstyle{remark}
\newtheorem{remark}[theorem]{Remark}

\begin{document}

\title{On the continuity of the geometric side of the trace formula}
\date{\today}
\begin{abstract}
We extend the geometric side of Arthur's non-invariant trace formula for a reductive group $G$ defined over $\Q$ continuously to a natural space $\funct(G(\A)^1)$ of test functions
which are not necessarily compactly supported.
The analogous result for the spectral side was obtained in \cite{MR2811597}. The geometric side is decomposed according to the following equivalence relation on $G(\Q)$:
$\gamma_1\sim\gamma_2$ if $\gamma_1$ and
$\gamma_2$ are conjugate in $G(\overline{\Q})$ and their semisimple parts are conjugate in $G(\Q)$. All terms in the
resulting decomposition are continuous linear forms on the space $\funct(G(\A)^1)$, and can be approximated (with continuous error terms) by naively truncated integrals.
\end{abstract}
\author{Tobias Finis}
\address{Universit\"{a}t Leipzig, Mathematisches Institut,
Postfach 100920, D-04009 Leipzig, Germany}
\email{finis@math.uni-leipzig.de}
\author{Erez Lapid}
\address{Department of Mathematics, Weizmann Institute of Science, Rehovot 7610001, Israel}
\thanks{Second named author partially supported by grant \#711733  from the Minerva Stiftung}
\email{erez.m.lapid@gmail.com}
\maketitle

\setcounter{tocdepth}{1}

\tableofcontents

\section{Introduction}
Let $G$ be a reductive group defined over $\Q$ and let $\A=\R\times\A_f$ be the ring of adeles.
As usual, we write $G(\A)^1=\cap\operatorname{Ker}\abs{\chi}_{\A^\times}$, where $\chi$ ranges over the rational characters of $G$.
The original (non-invariant) form of Arthur's trace formula is an identity between two distributions $f \mapsto J(f)$ on $G(\A)^1$,
a geometric one and a spectral one.
The geometric side can be split according to the following equivalence relation on $G(\Q)$:
$\gamma_1\sim\gamma_2$ if $\gamma_1$ and
$\gamma_2$ are conjugate in $G(\overline{\Q})$ and their semisimple parts are conjugate in $G(\Q)$.
In other words, if $\Orb$ is the set of pertinent equivalence classes, then there is a decomposition
\begin{equation} \label{ATF}
J(f)=\sum_{\orb\in\Orb}J_{\orb}(f),\ \ f\in C_c^\infty(G(\A)^1).
\end{equation}
(See \cite{MR835041}. Actually, in [ibid.] a finer equivalence relation is considered, but for our purposes the relation $\sim$ is more suitable.\footnote{In Remark \ref{FinerRelation} below we will consider a slight refinement of the equivalence relation $\sim$ in the case where the derived group of $G$ is not simply connected.})
The distributions $J_{\orb}(f)$ are well understood (as weighted orbital integrals) in the case where $\orb$
is a semisimple conjugacy class of $G(\Q)$. However, they are more mysterious for other classes,
most notably for the unipotent geometric orbits. See \cite{1411.3005, 1310.0541, 1412.8673} for some recent progress on this problem.

For any compact open subgroup $K$ of $G(\A_f)$ the space $G(\A)^1/K$ is a differentiable manifold
(namely a countable disjoint union of copies of $G(\R)^1=G(\R)\cap G(\A)^1$).
Any element $X\in\univ(\Lieg_\infty^1)$ of the universal enveloping algebra of the Lie algebra $\Lieg_\infty^1$
of $G(\R)^1$ defines a left invariant differential operator $f\mapsto f*X$ on $G(\A)^1/K$.
Let $\funct(G(\A)^1;K)$ be the space of smooth right $K$-invariant
functions on $G(\A)^1$ which belong, together with all their derivatives, to $L^1(G(\A)^1)$.
The space $\funct(G(\A)^1;K)$ becomes a Fr\'echet space under the seminorms
\[
\norm{f*X}_{L^1(G(\A)^1)}, \ \ X\in\univ(\Lieg_\infty^1).
\]
We denote by $\funct(G(\A)^1)$ the union of $\funct(G(\A)^1;K)$ as $K$ varies over the compact open subgroups of $G(\A_f)$
and endow $\funct(G(\A)^1)$ with the inductive limit topology.

The purpose of this paper is to show that the geometric side of Arthur's trace formula \eqref{ATF}
extends continuously to the class $\funct(G(\A)^1)$. More precisely, we show that
\[
\sum_{\orb\in\Orb}\abs{J_{\orb}(f)}
\]
extends to a continuous seminorm on $\funct(G(\A)^1)$
(see Corollary \ref{CorollaryContinuity} below). The analogous result for the spectral side was obtained in \cite{MR2811597}, so that the present paper establishes a trace formula for functions in the class $\funct(G(\A)^1)$.

Moreover, we show that the distributions $J_\orb$ can be computed using naive truncation. Namely, using the notation of \S\ref{sec: notation} below, there exist distributions
$f \mapsto J_{\orb}^T(f)$, $\orb\in\Orb$, on $\funct(G(\A)^1)$, which are polynomial functions of the parameter $T \in \aaa_0$,
and satisfy the following approximation property:
for any $K$ there exists a continuous seminorm $\mu$ on $\funct(G(\A)^1;K)$ (depending polynomially on the level of $K$) such that
\[
\sum_{\orb\in\Orb}\abs{\int_{G(\Q)\bs G(\A)^1_{\le T}}\sum_{\gamma\in\orb}f(x^{-1}\gamma x)\ dx-J_{\orb}^T(f)}
\le\mu(f)(1+\norm{T})^re^{-d(T)}
\]
for all $f \in \funct(G(\A)^1;K)$ and
$T\in\aaa_0$ with $d(T):=\min_{\alpha\in\Delta_0}\sprod{\alpha}T>d_0$, where the constants $d_0$ and $r$ are
independent of $f$ and $K$.
The distribution $J_{\orb}(f)$ is obtained by evaluating $J^T_{\orb}(f)$ at a certain distinguished point $T=T_0$
(see \cite[\S1]{MR625344}).

In a previous paper \cite{MR2801400} we proved similar results for the contribution of the semisimple conjugacy classes
of $G(\Q)$. In fact, if we coarsen the relation $\sim$ by only requiring that the semisimple parts are conjugate in $G(\Q)$
(obtaining the so-called coarse classes of Arthur), then the methods of [ibid.], combined with Arthur's basic procedure in \cite{MR518111}, yield the desired result for the coarse geometric expansion, as
will be explained in \S \ref{sec: basic}--\ref{sec: contcoarse} below.
The main result of this part of the paper is Theorem \ref{thm: maincoarse}.
To go further, we use recent work of
Chaudouard--Laumon \cite{MR3427596}, which provides a suitable definition for the modified kernel pertaining to a class of $\sim$.
This definition, which has also been suggested by Hoffmann
\cite{1412.8673}, turns out to be very useful for our purpose.
The continuity of the finer decomposition with respect to $\sim$
is dealt with in \S \ref{sec: modif}--\ref{sec: main}, the main results being Theorem \ref{thm: main} and Corollary
\ref{CorollaryContinuity}.
We remark that our results extend earlier results by Hoffmann in this direction \cite{MR2434856}.
In the Lie algebra case, Chaudouard proved very recently similar results for the space of Schwartz-Bruhat functions \cite{1510.02783}.

One of the main reasons to consider the trace formula on the space $\funct(G(\A)^1)$
is the connection to automorphic $L$-functions (in a suitable right half-plane).
Namely, fixing a model of $G$ over $\Z$, there exists a finite set $S_0 \supset \{ \infty \}$ of places of $\Q$ with the following property. Let $\rho$ be a representation of the $L$-group $^L{G}$ of $G$. Then for all $p \notin S_0$ there exists a unique
bi-$G(\Z_p)$-invariant function
$\phi_{\rho,p,s}$ on $G (\Q_p)$ with
$\tr \pi_p (\phi_{\rho,p,s}) = L_p (\pi_p, \rho, s)$ for all
unramified representations $\pi_p$ of $G(\Q_p)$, where both sides are either considered as formal power series in $p^{-s}$ or  $\Re s$ has to be suitably large.

Let now $S \supset S_0$ be finite set of places of $\Q$, $\phi_p
\in L^1 (G (\Q_p))$ for all $p \in S \setminus \{ \infty \}$ and $\phi_\infty$ be a $C^\infty$-function on $G(\R)$ with
$\norm{f*X}_{L^1(G(\R))} < \infty$ for all $X\in\univ(\Lieg_\infty)$.
Set
\[
\phi_{\rho,s} (g) = \prod_{v \in S} \phi_v (g_v) \prod_{p \notin S}
\phi_{\rho,p,s} (g_p), \quad g = (g_v) \in G (\A).
\]
Then for $\Re s$ large enough (depending on $G$ and $\rho$) the function
\[
f_{\rho,s}  (g) = \int_{S_Z(\R)^\circ} \phi_{\rho,s}  (a g) \ da,
\]
where $S_Z$ is the maximal split torus contained in the center of $G$,
is an element of $\funct(G(\A)^1)$. The contribution of a
discrete automorphic representation $\pi$ of $G (\A)$ to the spectral side
of the trace formula for $f_{\rho,s}$ will be non-zero only if $\pi$ is unramified outside of $S$, and in this case it will be equal to
\[
m (\pi) \prod_{v \in S} \, \tr \pi_v (\phi_v) \, L^S (\pi, \rho, s),
\]
where $m (\pi)$ is the multiplicity of $\pi$ in the discrete spectrum and $L^S (\pi, \rho, s) = \prod_{p \notin S} L_p (\pi_p, \rho, s)$ is the (incomplete) automorphic $L$-function of $\pi$ associated to $\rho$.

A prototype case is $G = \GL (n)$ and $\rho$ the standard representation. In this case one might
more concretely take $\phi$ to be the product of the restriction to $G (\A)$ of a Schwartz-Bruhat function $\Phi$ on the adelic Lie algebra
$\Lieg (\A)$ of $G$ and of the function $\abs{\det}_{\A^\times}^{s+(n-1)/2}$. The resulting function $f_{\rho,s}$ will be an element of
$\funct(G(\A)^1)$ for $\Re s > (n+1)/2$. The contribution of a
discrete automorphic representation $\pi$ can be expressed in terms of the zeta integrals of Godement-Jacquet \cite{MR0342495}, and it is therefore the product of a locally defined entire function of $s$ (which depends on $\Phi$ and $\pi$) and of the completed standard $L$-function of $\pi$ at the point $s$.
This case and its connection to the trace formula for the
Lie algebra have been studied by Jasmin Matz
(see \cite{1308.5394} and work in preparation).


Although this is very far-fetched at this stage, the hope is that using the trace formula for generating functions of the above type
will ultimately provide means to attack Langlands functoriality conjectures beyond the very limited scope (however important) of the current
methods.
This general idea, and its variations were suggested by Langlands in \cite{MR2058622, MR2317447} with some subsequent analysis in
\cite{MR2779866, MR3117742} -- see also \cite{MR3220933} and \cite{1412.6174} for closely related themes.
The humble purpose of the current paper is to provide one of the very first technical steps in this direction.

We are very grateful to Werner Hoffmann for spotting a mistake as well as a number of inaccuracies in an earlier version of this paper and for his suggestion to explicate
the dependence of our estimates on the level of $K$.
We also thank Laurent Clozel and Bao Ch{\^a}u Ng{\^o} for useful discussions.

\section{Notation and preliminaries}

\subsection{} \label{sec: notation}
For the rest of the paper let $G$ be a reductive group defined over $\Q$.
Let $G^{\der}$ be its derived group and $Z_G$ be the center of $G$.
We fix a minimal parabolic subgroup $P_0$ defined over $\Q$ and a Levi decomposition $P_0=M_0\ltimes N_0$ of $P_0$.
Let $S_0$ be the split part of the center of $M_0\cap G^{\der}$ and $X_*(S_0)$ the lattice of co-characters of $S_0$.
Let $A_0$ be the identity component of the topological group $S_0(\R)$ and $\aaa_0=X_*(S_0)\otimes\R$.
We can identify the dual space $\aaa_0^*$ with $X^*(M_0/Z_G)\otimes\R$, where $X^*(M_0/Z_G)$ is the lattice of rational characters of $M_0$ (or $P_0$)
which are trivial on $Z_G$.
We denote the set of simple roots of $S_0$ acting on $N_0$ by $\Delta_0$.
Let $\rho_0\in\aaa_0^*$ be the element corresponding to
$\modulus_0^{1/2}$, where $\modulus_0$ is the modulus function of $P_0$.
We define the homomorphism
\[
\Ht_0:M_0(\A)\rightarrow\aaa_0
\]
by $\sprod{\chi}{\Ht_0(m)}=\log\abs{\chi(m)}_{\A^*}$ for any $m\in M_0(\A)$ and $\chi\in X^*(M_0/Z_G)$, where $\abs{\cdot}_{\A^*}$
is the standard absolute value on $\A^*$.

Except otherwise mentioned, all parabolic subgroups considered are implicitly assumed to be defined over $\Q$.
If $P$ is a standard parabolic subgroup, then we write $P=M_P\ltimes N_P$ (or simply $P=M\ltimes N$, if $P$ is clear from the context) for its standard Levi decomposition.
The set of simple roots of $S_0$ in $N_0\cap M_P$ is denoted by $\Delta_0^P$. It is a subset of $\Delta_0$.
We write $\aaa_P=X_*(S_M)\otimes\R$, where $S_M$ is the split part of the center of $M\cap G^{\der}$, and view $\aaa_P$ as a subspace of $\aaa_0$
whose complement is $\aaa_0^P=X_*(S_0\cap M^{\der})$. Thus, we may view the dual space $\aaa_P^*$ as a subspace of $\aaa_0^*$.
We also write $A_M$ for the identity component of $S_M(\R)$.
We write $\Delta_P$ for the image of $\Delta_0\setminus\Delta_0^P$ under the projection $\aaa_0^*\rightarrow\aaa_P^*$.
More generally, if $Q$ is a parabolic subgroup containing $P$, then we write $\Delta_P^Q$ for the projection of
$\Delta_0^Q\setminus\Delta_0^P$ under $\aaa_0^*\rightarrow\aaa_P^*$.
Similarly, we have the set of coroots $\Delta_0^\vee$ and, more generally, for $Q \supset P$ the set $(\Delta_P^Q)^\vee$ which forms a basis of
$\aaa_P^Q:=\aaa_P\cap\aaa_0^Q$. We denote the basis
of $(\aaa_P^Q)^*$ (resp., $\aaa_P^Q$)
dual to $(\Delta_P^Q)^\vee$ (resp., $\Delta_P^Q$)
by $\hat\Delta_P^Q$ (resp., $(\hat\Delta_P^Q)^\vee$). As usual, we suppress the superscript if $Q=G$.
We write $\tau_P^Q$ and $\hat\tau_P^Q$ for the characteristic functions of the sets
\[
\{X\in\aaa_0:\sprod{\alpha}X>0\text{ for all }\alpha\in\Delta_P^Q\}
\]
and
\[
\{X\in\aaa_0:\sprod{\varpi}X>0\text{ for all }\varpi\in\hat\Delta_P^Q\},
\]
respectively.

We fix a ``good'' maximal compact subgroup $\K = \K_\infty \K_{\fin}$ of $G(\A)$
(i.e., we require $\K$ to be admissible relative to
$M_0$ in the sense of \cite[\S1]{MR625344}). We extend
the left $M_0 (\A)^1$-invariant map
$\Ht_0:M_0(\A)\rightarrow\aaa_0$
to a left $P_0(\A)^1$- and right $\K$-invariant function
\[
\Ht_0:G(\A)\rightarrow\aaa_0.
\]
For $T_1 \in \aaa_0$ let
\[
\siegel_{T_1}=\{x\in G(\A):\tau_0(\Ht_0(x)-T_1)=1\}
\]
and more generally
\[
\siegel_{T_1}^P=\{x\in G(\A):\tau_0^P(\Ht_0(x)-T_1)=1\}
\]
for any $P\supset P_0$.
These sets are then evidently left $P_0(\A)^1$-invariant.
By reduction theory, there exists $T_1\in\aaa_0$ such that
\[
P(\Q)\siegel_{T_1}^P=G(\A)
\]
for all $P\supset P_0$, and in particular for $P=G$.
Thus,
\begin{multline} \label{eq: bndsiegel}
\int_{P(\Q)\bs G(\A)^1}\abs{f(x)}\ dx\le\\\int_{\K}
\int_{N_0(\Q)\bs N_0(\A)}
\int_{A_0} \int_{M_0(\Q)\bs M_0(\A)^1}
\abs{f(uamk)}\tau_0^P(\Ht_0(a)-T_1)\modulus_0(a)^{-1}\ dm\ da\ du\ dk
\end{multline}
for any left $P(\Q)$-invariant measurable function $f$ on $G(\A)^1$.
We fix $T_1$ as above once and for all.

Let
\[
G(\A)^1_{\le T}=\{g\in G(\A)^1:\hat\tau_0(T-\Ht_0(\gamma g))=1\text{ for all }\gamma\in G(\Q)\}.
\]
There exists $d_0>0$ (which depends only on $G$, $P_0$ and
$\K$, and which we may therefore fix once and for all) such that
\[
G(\A)^1_{\le T}\cap\siegel_{T_1}=\{g\in G(\A)^1:\tau_0(\Ht_0(g)-T_1)\hat\tau_0(T-\Ht_0(g))=1\}
\]
provided that $d(T):=\min_{\alpha\in\Delta_0}\sprod{\alpha}T>d_0$.
More generally, for any $P\supset P_0$ let $F^P(x,T)$ be the characteristic function of the set
\[
\{g\in G(\A):\hat\tau_0^P(T-\Ht_0(\gamma g))=1\text{ for all }\gamma\in P(\Q)\}.
\]
By Arthur's partition lemma \cite[Lemma 6.4]{MR518111}, we have
\begin{equation} \label{eq: Arpar}
\sum_{P\supset P_0}\sum_{\gamma\in P(\Q)\bs G(\Q)}F^P(\gamma x,T)\tau_P(\Ht_P(\gamma x)-T)=1, \ \ x\in G(\A),
\end{equation}
provided that $d(T)>d_0$.

Let $W=\operatorname{N}_{G(\Q)}(M_0)/M_0$ be the Weyl group of $G$.
For any $w\in W$ we fix a representative $n_w\in G(\Q)$ (it is determined up to multiplication by an element of $M_0(\Q)$) and set
\begin{equation} \label{def: Q(w)}
Q(w)=\text{the smallest standard parabolic subgroup of $G$ containing $n_w$}.
\end{equation}

\subsection{}
Let $H$ be an algebraic subgroup of $G$ defined over $\Q$.
We denote by $\modulus_H$ the modulus function of the group $H(\A)$.
We will write $\hhh$ for the Lie algebra of $H(\R)$.
(We will retain this typographic convention for other groups.)
We recall the norms
\[
\norm{f}_{H,k}=\sum_i\norm{f*X_i}_{L^1(H(\A))},\ \ \ k\ge0,
\]
where $X_i$ ranges over a fixed basis of $\univ(\hhh)_{\le k}$ with respect to the standard filtration.
We write $\mu_0^H(f)=\norm{f}_{H,\dim H}$.
For any compact open subgroup $K$ of $H(\A_f)$ we consider the Fr\'echet space $\funct(H(\A);K)$ of right $K$-invariant smooth
functions $f$ on $H(\A)$ such that $\norm{f}_{H,k}<\infty$ for all $k$, and define $\funct(H(\A)^1;K)$ and
$\mu_0^{H,1}(f)=\norm{f}_{H(\A)^1,\, \dim H - \operatorname{rk} X^* (H)}$
analogously.

We recall a few useful facts about these norms. (See \cite[\S3]{MR2801400}.
Note that the dependence on $K$ is not explicated in [ibid.], but it is easy to extract it from the argument. Also note that a factor $\modulus_H(h)^{-1}$ is missing on the right-hand side of the second inequality of \cite[Lemma 3.3]{MR2801400}.)

Henceforth, for non-negative quantities $A$ and $B$ we use the notation $A\ll B$ to mean that there exists some constant $c>0$ such that
$A\le cB$. If $c$ depends on some additional parameters (such as $X$) we will write $A\ll_XB$.

\begin{lemma} \label{lem: sob}
\begin{enumerate}
\item \label{part: ltcsum}
For any $f\in\funct(H(\A)^1;K)$ we have
\[
\sum_{\gamma\in H(\Q)}\abs{f(\gamma)}\ll_{H} \vol (K)^{-1} \mu_0^{H,1}(f).
\]
\item \label{part: compact conjugation}
Let $C\subset H(\A)^1$ be a compact set and $\mu$ a continuous seminorm on $\funct(H(\A)^1;K)$.
Then $\sup_{x\in C}\mu(f(\cdot x))$ and $\sup_{x\in C}\mu(f(x^{-1}\cdot x))$
are continuous seminorms on the space $\funct(H(\A)^1;\cap_{x\in C}x^{-1}Kx)$
and $\norm{f(x^{-1}\cdot x)}_{H(\A)^1,k}=\norm{f(\cdot x)}_{H(\A)^1,k}\ll_{C,k}\norm{f}_{H(\A)^1,k}$ for all $f\in\funct(H(\A)^1)$, $k\ge0$, $x\in C$.
\item \label{part: pos}
For any $f\in\funct(G(\A)^1;K)$ there exists $\tilde f\in\funct(G(\A)^1;\K_{\fin})$ such that $\tilde f(x)\ge\abs{f(x)}$
for all $x\in G(\A)^1$, $\tilde f$ is right $\K$-invariant, and $\norm{\tilde f*X}_{L^1(G(\A)^1)}\ll_{X} \vol (K)^{-1} \mu_0^{G,1}(f)$ for any $X\in\univ(\Lieg^1)$.
\item \label{part: consterm}
Let $P$ be a standard parabolic subgroup of $G$.
For any $f\in\funct(G(\A)^1;K)$ define
\begin{equation} \label{eq: f_P}
f_P(m)=\int_{\K}\int_{N(\A)}f(k^{-1}mnk)\ dn \ dk, \ \ m\in M(\A)^1.
\end{equation}
Then $f\mapsto f_P$ is a continuous map from the space $\funct(G(\A)^1;K)$ to the space $\funct(M(\A)^1;\cap_{k\in\K}kKk^{-1}\cap M(\A))$.
\end{enumerate}
\end{lemma}

\subsection{} \label{subsectionlevel}
We fix a faithful $\Q$-rational representation $r_0: G \to \GL (N_0)$ such that $\K_{\fin} = \{ g \in G (\A_f) \, : \, r_0 (g)
\in \GL (N_0, \hat{\Z}) \}$.
For any positive integer $N$ let
\[
\K (N) =\{ g \in G (\A_f) \, : \, r_0 (g) \equiv 1 \pmod{N} \}
\]
be the principal congruence subgroup of level $N$, a factorizable normal open subgroup of $\K_{\fin}$.
The groups $\K (N)$ form a neighborhood base of the identity element in $G(\A_f)$.

Throughout the paper $K$ denotes a compact open subgroup of $G(\A_f)$. The level of $K$ is defined as the smallest positive integer $N$ with $\K (N) \subset K$. We denote it by $\level (K)$.

\section{An estimate for truncated integrals} \label{sec: basic}

\subsection{} In this section, we prove a slight variant of the main result of \cite{MR2801400}, which is basic for all following estimates.
For any parabolic subgroup $P$ of $G$ (defined over $\Q$) we
define
\[
G(\Q)\rg{P}=G(\Q)\setminus\cup_{P\subset P'\subsetneq G}P'(\Q).
\]
The set $G(\Q)\rg{P}$ is bi-$P(\Q)$-invariant and $G(\Q)\rg{G}=G(\Q)$.
For standard parabolic subgroups $P\subset Q$ we set
\begin{equation} \label{def: xi}
\smofrts PQ=\sum_{\alpha\in\Delta_0^Q\setminus\Delta_0^P}\alpha\in\aaa_0^*.
\end{equation}

\begin{theorem} \label{thm: analog}
There exist an integer $r\ge0$, depending only on $G$,
and a continuous seminorm $\mu$ on $\funct(G(\A)^1;K)$, such that
for any standard parabolic subgroup $P$ of $G$ and any $l\ge0$ we have
\begin{multline} \label{eq: mainbnd11}
\int_{P(\Q)\bs G(\A)^1}F^P(g,T)\tau_P(\Ht_P(g)-T)\norm{\Ht_P(g)-T_P}^l\sum_{\gamma\in G(\Q)\rg{P}}
\abs{f(g^{-1}\gamma g)}\ dg\\\ll_{l}(1+\norm{T})^re^{-\sprod{\smofrts P{}}T}\mu(f)
\end{multline}
for any $f\in\funct(G(\A)^1;K)$ and any $T \in \aaa_0$ such that $\sprod{\alpha}{T-T_1} \ge 0$
for all $\alpha \in \Delta_0$.
Moreover, we can take $\mu = c \vol (K)^{-1} \mu_0^{G,1}$ with a constant $c$ that does not depend on $K$.
\end{theorem}

For $P = P_0$, $T=T_1$ and $l=0$ this specializes to one of the main intermediate results of \cite{MR2801400}
(which implies immediately the continuity of the regular elliptic contribution to the trace formula).
For $P=G$ we obtain that
\begin{multline} \label{eq: P=G}
f\mapsto\sup_{T:d(T)>d_0}(1+\norm{T})^{-r}\int_{G(\Q)\bs G(\A)^1_{\le T}}\sum_{\gamma\in G(\Q)}\abs{f(g^{-1}\gamma g)}\ dg\\
\text{is a continuous seminorm on $\funct(G(\A)^1)$.}
\end{multline}
(A more careful analysis shows that we can in fact take $r = \dim \aaa_0$.)

In the remainder of this section we will prove Theorem \ref{thm: analog}.
The proof follows the argument of \cite{MR2801400} closely, but on the one hand it is possible to simplify the argument (cf.~[ibid., Remark 3]),
and on the other hand we need to keep track of the dependence on $T$.

As in [ibid.], the main intermediate step is an estimate for truncated integrals over the Bruhat cells of Weyl group elements $w \in W$ with $Q(w)P=G$. We state the necessary
generalization of [ibid., Proposition 5.1] now, and postpone the proof to \S \ref{ProofProposition} below.

\begin{proposition} \label{prop: modprop51}
Let $P$ be a standard parabolic subgroup of $G$ and $w\in W$
with $Q(w)P=G$.
There exist an integer $r\ge0$ and
a continuous seminorm $\mu$ on $\funct(G(\A)^1;K)$ such that for any $l\ge0$ we have
\begin{multline*}
\int_{N_0(\A)/N_w(\A)}\int_{N_0(\A)}\int_{A_0}\int_{M_0(\A)^1}\abs{f(a^{-1}u_2n_wau_1m)}\chi_{T,P,l}(a)\ dm\ da\ du_1\ du_2\\\ll_{K,l}
\mu(f)(1+\norm{T})^re^{-\sprod{\smofrts P{}}T},
\end{multline*}
where $N_w=N_0\cap n_wN_0n_w^{-1}$ and
\[
\chi_{T,P,l}(a)=\tau_P(\Ht_P(a)-T)\tau_0^P(\Ht_0(a)-T_1)\hat\tau_0^P(T-\Ht_0(a))\sum_{\alpha\in\Delta_P}\sprod{\alpha}{\Ht_P(a)-T}^l.
\]
\end{proposition}


As usual, we also need to estimate sums over the unipotent radicals of standard parabolic subgroups by integrals.

\begin{lemma} \label{lem: latticesum}
Let $P=M\ltimes N$ be a standard parabolic subgroup of $G$.
Then there exist $X_1,\dots,X_m\in\univ(\nnn)$ such that for any compact open subgroup $K'$ of $N(\A_f)$ we have
\[
\sum_{n\in N(\Q)}\abs{f(a^{-1}nu a)}\le \vol (K')^{-1} \sum_i\int_{N(\A)}\abs{(f*X_i)(a^{-1}na)}\ dn
\]
for any $f\in\funct(N(\A);K')$, $u\in N(\A)$ and $a\in A_0$ such that $\tau_0(\Ht_0(a)-T_1)=1$.
\end{lemma}

\begin{proof}
The special case $a=1$ follows immediately from Lemma \ref{lem: sob}, parts \ref{part: ltcsum} and \ref{part: compact conjugation},
since $u$ can be taken in a compact set.

Moreover, we can take the differential operators $X_i$ to be a basis for $\univ(\nnn)_{\le\dim\nnn}$ with $\Ad(a)X_i=\chi_i(a)X_i$
for $a\in A_0$, where each $\chi_i$ is a character of $A_0$ which is a sum of positive roots.
We therefore have $\abs{\chi_i(a)}^{-1}\ll 1$ for $\tau_0(\Ht_0(a)-T_1)=1$.
The lemma follows, since the function $f_a=f(a^{-1}\cdot a)$
on $N (\A)$ satisfies
\[
f_a*X_i(n)=[f*\Ad(a^{-1})X_i](a^{-1}na)
= \chi_i(a)^{-1} (f*X_i) (a^{-1} na),
\ \ n\in N(\A).\qedhere
\]
\end{proof}

\begin{proof}[Proof of Theorem \ref{thm: analog}]
Using \eqref{eq: bndsiegel}, we first estimate the left-hand side of \eqref{eq: mainbnd11} by a constant multiple (which depends only on $l$) of
\[
\int_{\K}\int_{N_0(\Q)\bs N_0(\A)}\int_{A_0}\int_{M_0(\Q)\bs M_0(\A)^1}
\sum_{\gamma\in G(\Q)\rg{P}}\abs{f((uamk)^{-1}\gamma uamk)}\chi(a)\modulus_0(a)^{-1}\ dm\ da\ du\ dk,
\]
where for convenience we write $\chi(a)=\chi_{T,P,l}(a)$.
Note that $\chi(a)$ is non-negative, and that under our assumption on $T$ the argument of \cite[pp. 943--944]{MR518111} shows that
\begin{equation} \label{supportchi}
\chi (a) > 0 \text{ implies } \tau_0(\Ht_0(a)-T_1)=1.
\end{equation}

Since $m$ and $k$ are integrated over compact sets, we can use Lemma \ref{lem: sob}, part \ref{part: compact conjugation},
to reduce to bounding
\[
\int_{N_0(\Q)\bs N_0(\A)}\int_{A_0}\sum_{\gamma\in G(\Q)\rg{P}}\abs{f(a^{-1}u^{-1}\gamma ua)}\chi(a)\modulus_0(a)^{-1}\ da\ du.
\]
Recall that $G(\Q)\rg{P}$ is bi-$P(\Q)$-invariant
and therefore a union of Bruhat cells.
In fact,
\[
G(\Q)\rg{P}=\bigcup_{w\in W: \, Q(w)P = G}
N_0(\Q)n_wP_0(\Q).
\]
Thus, we need to consider
\[
\int_{N_0(\Q)\bs N_0(\A)}\int_{A_0}\sum_{\gamma
\in N_0(\Q)n_wP_0(\Q)}
\abs{f(a^{-1}u^{-1}\gamma ua)}\chi(a)\modulus_0(a)^{-1}\ da\ du
\]
for any $w \in W$ with $Q(w)P = G$. We write this as
\[
\int_{N_0(\Q)\bs N_0(\A)}\int_{A_0}\sum_{u_2\in N_w(\Q)\bs N_0(\Q)}\sum_{m\in M_0(\Q)}\sum_{u_1\in N_0(\Q)}
\abs{f(a^{-1}u^{-1}u_2^{-1}mn_wu_1ua)}\chi(a)\modulus_0(a)^{-1}\ da\ du.
\]
Using \eqref{supportchi}, we can apply
Lemma \ref{lem: latticesum} and estimate the sum over $u_1$ by
the integrals of the functions $f*X$, $X$ ranging over a fixed finite set of differential operators.
Replacing $f$ by one of these derivatives, we can reduce to
\[
\int_{N_0(\Q)\bs N_0(\A)}\int_{A_0}\int_{N_0(\A)}\sum_{u_2\in N_w(\Q)\bs N_0(\Q)}\sum_{m\in M_0(\Q)}
\abs{f(a^{-1}u^{-1}u_2^{-1}mn_wau_1)}\chi(a)\ du_1\ da\ du,
\]
i.e., to
\[
\int_{N_w(\Q)\bs N_0(\A)}\int_{A_0}\int_{N_0(\A)}\sum_{m\in M_0(\Q)}\abs{f(a^{-1}u^{-1}mn_wau_1)}\chi(a)\ du_1\ da\ du.
\]
Note that as a function of $u$, the inner integral is left $N_w(\A)$-invariant, and hence we get
\[
\int_{N_w(\A)\bs N_0(\A)}\int_{A_0}\int_{N_0(\A)}\sum_{m\in M_0(\Q)}\abs{f(a^{-1}u^{-1}mn_wau_1)}\chi(a)\ du_1\ da\ du,
\]
which is also
\[
\int_{N_w(\A)\bs N_0(\A)}\int_{A_0}\int_{N_0(\A)}\sum_{m\in M_0(\Q)}\abs{f(a^{-1}u^{-1}n_wau_1m)}\chi(a)\ du_1\ da\ du.
\]
Finally, using Lemma \ref{lem: sob}, part \ref{part: ltcsum}, we reduce to
\[
\int_{N_w(\A)\bs N_0(\A)}\int_{A_0}\int_{N_0(\A)}\int_{M_0(\A)^1}\abs{f(a^{-1}u^{-1}n_wau_1m)}\chi(a)\ dm\ du_1\ da\ du,
\]
which is continuous by Proposition \ref{prop: modprop51}. The assertion about
$\mu$ follows directly from Lemma \ref{lem: sob}, part \ref{part: pos}.
\end{proof}

\subsection{} \label{ProofProposition}
It remains to prove Proposition \ref{prop: modprop51}. Since the argument is very similar to the proof of
\cite[Proposition 5.1]{MR2801400}, we refer the reader to the earlier paper and
only give the parts of the argument of [ibid.] that need to be modified.
We remark that the case $P = P_0$, $l = 0$ and $T = T_1$ is already contained in [ibid., Proposition 5.1].
The main difference is that we now have to keep track of the dependence of $T$.

To that end, we first recall \cite[Proposition 3.1]{MR2801400}.
Let $V$ be a finite-dimensional real vector space and let $\Diff(V)$ be the space of invariant differential operators on $V$
with the standard filtration.
Let $\funct(V)$ be the Fr\'echet space of smooth functions $f$ on $V$ such that $\norm{f*D}_{L^1(V)}<\infty$
for any $D\in\Diff(V)$.
For any $f\in C_c^\infty(V)$ let $\hat f$ be its Fourier-Laplace transform given by
\[
\hat f(\lambda)=\int_Ve^{-\sprod{\lambda}v}f(v)\ dv,\ \ \lambda\in V^*_{\C},
\]
where $V^*_{\C}$ the complexified dual space of $V$.
Then $\hat f$ is an entire function which is rapidly decreasing for $\Re\lambda$ in a compact set.

Fix a linearly independent set $S$ in $V$ and $\mu_0\in V^*$.
Let $h$ be a holomorphic function on
the set of $\lambda \in V^*_{\C}$ with
$\Re\lambda\in\mathcal{R}$, where $\mathcal{R}$ is a bounded connected open subset of $V^*$ containing $\mu_0$.
Assume that $h$ is majorized by a polynomial function and let $\lambda_0\in\mathcal{R}$ be with $\sprod{\lambda_0-\mu_0}u>0$ for all $u\in S$. Then
\[
f\in C_c^\infty(V)\mapsto\int_{\Re\lambda=\lambda_0}\frac{\hat f(\lambda-\mu_0)h(\lambda)}{\prod_{u\in S}\sprod{\lambda-\mu_0}u}\ d\lambda
\]
extends to a continuous functional on $\funct(V)$.
We now make this statement effective as follows.

\begin{proposition} \label{prop: eff 3.1}
Let $S$, $\mu_0$, $\lambda_0$, and $h$ be as above.
For any $n\ge0$ there exists a continuous seminorm $\mu$ on $\funct(V)$ such that
\[
\abs{\int_{\Re\lambda=\lambda_0}\frac{\hat f(\lambda-\mu_0)h(\lambda)}{\prod_{u\in S}\sprod{\lambda-\mu_0}u}\ d\lambda}\ll_{n,S}
\mu(f)\sum_i\sup_{\Re\lambda=\mu_0}(1+\norm{\lambda})^{-n}\abs{(h*X_i)(\lambda)},
\]
where $(X_i)_i$ is a basis of $\Diff(V^*)_{\le\abs{S}}$.
\end{proposition}

\begin{proof}
We may assume without loss of generality that $\mu_0=0$.
Following the proof of
\cite[Proposition 3.1]{MR2801400},
let $\varpi_u \in V^*$, $u \in S$, be elements with
$\sprod{\varpi_u}{u'} = \delta_{u,u'}$ and
define for any $I\subset S$ the holomorphic function $h_{S,I}$ by
\[
h_{S,I}(\lambda)=\frac{\sum_{I\subset J\subset S}(-1)^{\abs{J}-\abs{I}}h(\lambda - \sum_{u \in J} \sprod{\lambda}u \varpi_u)}{\prod_{u\in S\setminus I}\sprod{\lambda}u}.
\]
We then need to estimate
$(1+ \norm{\lambda})^{-n} \abs{h_{S,I} (\lambda)}$
for $\lambda \in \iii I^\perp$.

Let $f$ be a smooth function on $\R$ and
$g(x)=(f(x)-f(0))/x$. Then we have
\[
\abs{g (x)} =
\abs{\int_0^1f'(tx)\ dt}
\le\sup_{\abs{t}\le\abs{x}}\abs{f'(t)}
\]
for all $x \in \R$.
Applying this to the independent variables
$\sprod{\lambda}{u}$, $u\in S\setminus I$, we obtain the estimate
\[
\abs{h_{S,I}(\lambda)}\ll\sum_i\sup_{\mu\in\iii V^*:\norm{\mu}\le\norm{\lambda}}\abs{(h*X_i)(\mu)}, \quad \Re\lambda=0,
\]
which finishes the proof.
%
\end{proof}


\begin{lemma} \label{LemmaConeIntegral}
Let $P$ be a standard parabolic subgroup of $G$
and $w \in W$ with $Q(w)P=G$. Then the integral
\[
\phi_{T,P,l}(\lambda)=\int_{A_0}a^{w^{-1}\lambda-\lambda}\chi_{T,P,l}(a)\ da
\]
converges absolutely and uniformly for $\Re\lambda$ in any compact subset of the positive Weyl chamber.
Moreover, for $\Re\lambda=\rho_0$ we have
\[
\abs{(\phi_{T,P,l}*D)(\lambda)}\ll_{D,l}
(1+\norm{T})^{d+ \dim \aaa^P_0} e^{-\sprod{\smofrts P{}}T}
\]
for any differential operator $D\in\Diff(\aaa_0^*)$ of degree $d$.
\end{lemma}

\begin{proof}
Using the direct sum decomposition $\aaa_0 =
\aaa_P \oplus \aaa^P_0$, we first apply Fubini's theorem formally to obtain
$\phi_{T,P,l} (\lambda) = \psi_{T,P,l} (\lambda)
\psi^P_{T} (\lambda)$ with
\[
\psi_{T,P,l} (\lambda)
=\int_{\aaa_P}e^{\sprod{w^{-1} \lambda - \lambda} X}\tau_P(X-T)\sum_{\alpha\in\Delta_P}\sprod{\alpha}{X-T}^l\ dX
\]
and
\[
\psi^P_{T} (\lambda) =
\int_{\aaa_0^P} e^{\sprod{w^{-1} \lambda - \lambda} Y}
\tau_0^P(Y-T_1)\hat\tau_0^P(T-Y)\ dY.
\]
The integrand in the definition of $\psi^P_{T} (\lambda)$
is compactly supported, and the integral therefore converges absolutely for any value of $\lambda$.
For $\lambda = \lambda_0 \in \aaa_0^*$, the integrands above are all non-negative real, and it remains to check the convergence of the integral defining $\psi_{T,P,l} (\lambda_0)$
for $\lambda_0$ in the positive Weyl chamber.
By \cite[Lemma 2.2]{MR2801400}, for such values of
$\lambda_0$ we have
$\lambda_0-w^{-1}\lambda_0=\sum_{\alpha\in\Delta_0^{Q(w)}}c_\alpha\alpha$ with $c_\alpha>0$ for all
$\alpha\in\Delta^{Q(w)}_0$.
Since $\Delta_0^{Q(w)}\cup\Delta_0^P=\Delta_0$ by our
assumptions on $P$ and $w$, we have
$\sprod{\lambda_0-w^{-1}\lambda_0}{\varpi^\vee}>0\text{ for all }\varpi^\vee\in\hat\Delta_P^\vee$.
The convergence assertion follows.

Furthermore, in the range of absolute convergence for $\psi_{T,P,l} (\lambda)$ we can express $X$ in the basis $\hat\Delta_P^\vee$ and compute
\[
\psi_{T,P,l} (\lambda) =
\vol(\aaa_P/\Z\hat\Delta_P^\vee)e^{\sprod{w^{-1}\lambda - \lambda}{T_P}}\prod_{\varpi^\vee\in\hat\Delta_P^\vee}\frac 1{\sprod{w^{-1}\lambda - \lambda}{\varpi^\vee}}
\sum_{\varpi^\vee\in\hat\Delta_P^\vee}\frac{l!}{\sprod{w^{-1}\lambda - \lambda}{\varpi^\vee}^l}.
\]

To estimate the derivatives of $\phi_{T,P,l}(\lambda)$
for $\Re\lambda = \rho_0$,
we may without loss of generality assume that $D = D_P D^P$
with $D_P \in\Diff(\aaa_P^*)$ and
$D^P \in\Diff((\aaa^P_0)^*)$ of degree $d_P$ and $d^P$, respectively. It is then clear from the expression above that
\[
\abs{(\psi_{T,P,l}*D_P) (\lambda)} \ll_{D_P, l}
(1 + \norm{T})^{d_P}
e^{\sprod{w^{-1}\rho_0 - \rho_0}{T_P}}.
\]
On the other hand, the function $\tau_0^P(Y-T_1)\hat\tau_0^P(T-Y)$ is the
characteristic function of the convex hull of the points
$\{T_Q^P + T_1^Q \, : \, P_0\subset Q\subset P\}$.
Therefore,
\[
\left|(\psi_T^P* D^P)(\lambda)\right|\ll_{D^P}(1+\lVert T \rVert)^{d^P +
\operatorname{dim}\mathfrak{a}_0^P}
\sum_{P_0\subset Q\subset P}e^{\left\langle w^{-1}\rho_0-\rho_0, T_Q^P
\right\rangle}.
\]
Since for any $P_0\subset Q\subset P$ we have
$\left\langle w^{-1}\rho_0-\rho_0, T_Q \right\rangle \le - \left\langle\xi_P, T \right\rangle$, the required estimate follows.
\end{proof}

It is of course possible to evaluate the integral $\phi_{T,P,l} (\lambda)$ explicitly, since the factor $\psi^P_{T} (\lambda)$
can be computed by applying \cite[Lemma 2.2]{MR625344}.

\begin{proof}[Proof of Proposition \ref{prop: modprop51}]
Following \cite[\S 5]{MR2801400}, we may assume $f
\in \funct (G(\A)^1)$ to be compactly supported, non-negative and right $\K$-invariant,
and write the integral as
\[
\int_{\Re\lambda = \lambda_0}
\phi_{T,P,l} (\lambda) m (w^{-1}, \lambda) \varphi(\lambda)
\, d \lambda
\]
for $\lambda_0 \in \aaa_0^*$ such that
$\lambda_0 - \rho_0$ lies in the positive Weyl chamber.
Here, the scalar $m (w^{-1}, \lambda)$ is the spherical intertwining operator (cf. [ibid., \S 3.3]) and
\[
\varphi (\lambda)
= \int_{A_0} \int_{P_0 (\A)^1} f (pa) a^{-(\lambda+\rho_0)}
\ dp \ da.
\]
It remains to apply Proposition \ref{prop: eff 3.1} with
$V = \aaa_0$, $\mu_0 = \rho_0$ and $S = \{ \alpha^\vee \in \Delta^\vee _0: \, w^{-1} (\alpha) < 0 \}$, and to invoke the estimate
of Lemma \ref{LemmaConeIntegral}.
\end{proof}

\section{Alternating sum-integrals over unipotent radicals} \label{sec: auxil}

\subsection{}
Let $I$ be a finite set and $\prmlvl\ge1$ an integer parameter (which will eventually be taken to be essentially the level of $K$).
For any $I'\subset I$ let $[0,\prmlvl]^{I;I'}$ be the face of the cube $[0,\prmlvl]^I$
consisting of the vectors whose coordinates in $I'$ vanish, endowed with the normalized Lebesgue measure.
Using integration by parts it is easy to see that
\[
\int_{[0,\prmlvl]^I}\frac{\partial^{\abs{I'}}f}{\prod_{i\in I'}\partial x_i}(x)\prod_{i\in I'}(x_i-\prmlvl)\ dx=
\sum_{I''\subset I'}(-1)^{\abs{I'\setminus I''}}\int_{[0,\prmlvl]^{I;I''}}f(x)\ dx,
\]
or equivalently,
\[
\int_{[0,\prmlvl]^{I;I'}}f(x)\ dx=
\sum_{I''\subset I'}\int_{[0,\prmlvl]^I}\frac{\partial^{\abs{I''}}f}{\prod_{i\in I''}\partial x_i}(x)\prod_{i\in I''}(x_i-\prmlvl)\ dx
\]
for any $f\in C^{\abs{I}}([0,\prmlvl]^I)$ and $I'\subset I$.
Thus, given any numbers $c_{I'}\in\C$ indexed by the subsets $I'$ of $I$, we have
\begin{equation} \label{eq: gencoefip}
\sum_{I'\subset I}c_{I'}\int_{[0,\prmlvl]^{I;I'}}f(x)\ dx=
\sum_{I'\subset I}d_{I'}\int_{[0,\prmlvl]^I}\frac{\partial^{\abs{I'}}f}{\prod_{i\in I'}\partial x_i}(x)\prod_{i\in I'}(x_i-\prmlvl)\ dx,
\end{equation}
where $d_{I'}=\sum_{I''\supset I'}c_{I''}$.
We single out a special case.

\begin{lemma} \label{lem: altsumvscls}
Let $I_j$, $j\in J$, be a family of (not necessarily disjoint) non-empty subsets of $I$.
For any $J'\subset J$ let $[0,\prmlvl]^I_{J'}$
be the face of $[0,\prmlvl]^I$
consisting of the vectors whose support is contained in the index set $\cup_{j\notin J'}I_j \subset I$,
endowed with the normalized Lebesgue measure.
Then
for any $f\in C^{\abs{I}}([0,\prmlvl]^I)$ we have
\[
\sum_{J'\subset J}(-1)^{\abs{J\setminus J'}}\int_{[0,\prmlvl]^I_{J'}}f(x)\ dx=
\sideset{}{'}\sum_{I'}\int_{[0,\prmlvl]^I}\frac{\partial^{\abs{I'}}f}{\prod_{i\in I'}\partial x_i}(x)\prod_{i\in I'}(x_i-\prmlvl)\ dx,
\]
and in particular
\begin{equation} \label{eq: sumJ}
\abs{\sum_{J'\subset J}(-1)^{\abs{J'}}\int_{[0,\prmlvl]^I_{J'}}f(x)\ dx}\le\prmlvl^{\abs{I}}
\sideset{}{'}\sum_{I'}\int_{[0,\prmlvl]^I}\abs{\frac{\partial^{\abs{I'}}f}{\prod_{i\in I'}\partial x_i}(x)}\ dx,
\end{equation}
where the sums on the right-hand sides range over the subsets $I'\subset I$ such that $I'\cap I_j\ne\emptyset$ for all $j\in J$.
\end{lemma}


\begin{proof}
Indeed, we take
\[
c_{I'}=\sum_{J'\subset J:\cup_{j\notin J'}I_j=I\setminus I'}(-1)^{\abs{J\setminus J'}},
\]
and note that
\[
d_{I'}=\sum_{I''\supset I'}c_{I''}=\sum_{J'\subset J:\cup_{j\notin J'}I_j\subset I\setminus I'}(-1)^{\abs{J\setminus J'}}=
\sum_{\{j\in J:I_j\cap I'\ne\emptyset\}\subset J'\subset J}(-1)^{\abs{J\setminus J'}},
\]
which is $1$ if $I_j\cap I'\ne\emptyset$ for all $j\in J$, and $0$ otherwise.
\end{proof}

As an immediate consequence we derive an adelic version as follows.
Let $\fund_{\fin}(\prmlvl)$ be the compact open subgroup
\[
\fund_{\fin}(\prmlvl)=\prod_{p<\infty}p^{v_p(\prmlvl)}\Z_p
\]
of $\A_f$ and let $\fund(\prmlvl)$ be the set
\[
\fund(\prmlvl)=[0,\prmlvl)\times\fund_{\fin}(\prmlvl),
\]
endowed with the product of the Lebesgue measure and the Haar measure, normalized such that $\vol(\fund(\prmlvl))=1$.
The set $\fund(\prmlvl)$ is a fundamental domain for $\Q\bs\A$.

\begin{lemma} \label{lem: altsumvs}
Let $I_j$, $j\in J$, be as in Lemma \ref{lem: altsumvscls}.
Then for any $f\in C^{\abs{I}}(\A^I;\fund_{\fin}(\prmlvl)^I)$ we have
\[
\abs{\sum_{J'\subset J}(-1)^{\abs{J'}}\int_{\fund^I_{J'}(\prmlvl)}f(x)\ dx}\le\prmlvl^{\abs{I}}
\sideset{}{'}\sum_{I'}\int_{\fund(\prmlvl)^I}\abs{\frac{\partial^{\abs{I'}}f}{\prod_{i\in I'}\partial x_i}(x)}\ dx,
\]
where $\fund(\prmlvl)^I_{J'}$ is the subset of $\fund(\prmlvl)^I$ consisting of the vectors whose coordinates outside
$\cup_{j\notin J'}I_j$ vanish (with the natural measure normalized by $\vol=1$).
\end{lemma}


\subsection{}
We fix once and for all a basis $(e_\alpha)_{\alpha\in\roots_0}$ of $\nnn_0$, indexed by a set $\roots_0$, such that $\Ad(a)$, $a \in A_0$,
acts on each basis vector $e_\alpha$ by multiplication with a character. We simply write
$\Ad(a)e_\alpha=\alpha(a)e_\alpha$ for all $a\in A_0$, $\alpha\in\roots_0$, i.e., we consider the index set
$\roots_0$ as the set of roots of $A_0$ on $N_0$, counting multiplicities.
For $\alpha,\beta\in\roots_0$ we write $\alpha\prec\beta$ if $\beta-\alpha=\sum_{\gamma\in\Delta_0}x_\gamma\gamma$
where $x_\gamma\ge0$ for all $\gamma$.
For any standard parabolic subgroup $P\subset G$ we view the set $\roots_P$ of roots of $A_0$ on $\nnn$ as a subset of $\roots_0$.
The vectors $(e_\alpha)_{\alpha\in\roots_P}$ form then a basis of $\nnn$.
Let $\coor_P:\A^{\roots_P}\rightarrow\nnn(\A)$ be the isomorphism given by $\coor_P ((x_\alpha)_{\alpha\in\roots_P})=\sum x_\alpha e_\alpha$.
We define
\[
\fund_P(\prmlvl)=\exp(\coor_P(\fund(\prmlvl)^{\roots_P})),\ \ \fund_{P,\fin}(\prmlvl)=\exp(\coor_P(\fund_{\fin}(\prmlvl)^{\roots_P})).
\]
Note that there exists an integer $\prmlvl_0\ge1$, depending only on $G$, such that $\fund_{P,\fin}(\prmlvl)$ is a compact open subgroup of $N(\A_f)$
whenever $\prmlvl$ is divisible by $\prmlvl_0$. Also, $\fund_P(\prmlvl)=\fund_0(\prmlvl)\cap N(\A)$, and more generally
$\fund_Q(\prmlvl)=\fund_P(\prmlvl)\cap N_Q(\A)$ for any $P\subset Q$.

\begin{lemma} \label{lem: fundom}
The set $\fund_P(\prmlvl)$ is a fundamental domain for $N(\Q)\bs N(\A)$.
\end{lemma}

\begin{proof}
Fix a linear order $\le$ on $\roots_P$ which extends $\prec$.
Let $N_{\ge\alpha}$ (resp., $N_{>\alpha}$) be the image under $\exp$ of the linear span of $e_\beta$, $\beta\ge\alpha$ (resp., $\beta>\alpha$).
Then $N_{\ge\alpha}$ and $N_{>\alpha}$ are normal subgroups of $N$ defined over $\Q$,
and for all $\alpha\in\roots_P$, $N_{\ge\alpha}/N_{>\alpha}$ is one-dimensional and central in $N/N_{>\alpha}$.
Moreover, $\exp(x+y)\in\exp(x)\exp(y)N_{>\alpha}$ for any $x\in\nnn$ and $y$ in the linear span of $e_\beta$, $\beta\ge\alpha$.
We show by induction on $\alpha$ that
\begin{equation} \label{eq: npi}
N(\Q)N_{\ge\alpha}(\A)\exp(\{\sum_{\beta<\alpha}x_\beta e_\beta:x_\beta\in\fund(\prmlvl)\})=N(\A).
\end{equation}
The case where $\alpha$ is the minimal element of $\roots_P$ is trivial. Assume that \eqref{eq: npi} holds for some $\alpha\in\roots_P$. Then
\begin{eqnarray*}
N(\A) &=& N(\Q)N_{\ge\alpha}(\A)\exp(\{\sum_{\beta<\alpha}x_\beta e_\beta:x_\beta\in\fund(\prmlvl)\})\\ & = &
N(\Q)N_{>\alpha}(\A)N_{\ge\alpha}(\Q)\exp(\{xe_\alpha:x\in\fund(\prmlvl)\})\exp(\{\sum_{\beta<\alpha}x_\beta e_\beta:x_\beta\in\fund(\prmlvl)\})\\
& = &
N(\Q)N_{>\alpha}(\A)\exp(\{xe_\alpha:x\in\fund(\prmlvl)\})\exp(\{\sum_{\beta<\alpha}x_\beta e_\beta:x_\beta\in\fund(\prmlvl)\})\\
& = &
N(\Q)N_{>\alpha}(\A)\exp(\{\sum_{\beta\le\alpha}x_\beta e_\beta:x_\beta\in\fund(\prmlvl)\}).
\end{eqnarray*}
This yields the induction step. Also, for the maximal $\alpha\in\roots_P$ we infer that $N(\Q)\fund_P(\prmlvl)=N(\A)$.

Suppose that $e\ne\gamma\in N(\Q)$ and write $\gamma=\exp(\coor((\lambda_\alpha)_{\alpha\in\roots_P}))$ with $\lambda_\alpha\in\Q$.
Let $\alpha$ be the smallest element of $\roots_P$ such that $\lambda_\alpha\ne0$.
Then $\gamma\in N_{\ge\alpha}$ and we have
\[
N_{>\alpha}(\A)\gamma\fund_P(\prmlvl)=N_{>\alpha}(\A)\exp(\{\lambda_\alpha+\sum_{\beta\le\alpha}x_\beta e_\beta:x_\beta\in\fund(\prmlvl)\}).
\]
Thus, $N_{>\alpha}(\A)\gamma\fund_P(\prmlvl)\cap N_{>\alpha}(\A)\fund_P(\prmlvl)=\emptyset$ and in particular,
$\gamma\fund_P(\prmlvl)\cap\fund_P(\prmlvl)=\emptyset$.
\end{proof}

In order to apply Lemma \ref{lem: altsumvs} to the alternating sum of constant terms, we will need to pass between partial derivatives
of $f\circ\exp$ and derivatives of $f$ for any $f\in C^\infty(N(\A))$. This is quite standard. 
For any $x\in\nnn$ consider the map $\phi_x:\nnn\rightarrow\nnn$ given by $\phi_x(y)=\log(\exp(-x)\exp(x+y))$.
For any $y\in\nnn$ we have
\[
\frac{\partial}{\partial y}(f\circ\exp)(x)=(f*D\phi_x(y))(\exp x),
\]
where $D\phi_x$ is the differential of $\phi_x$ at $0$, considered as a linear transformation from $\nnn$ to itself.
By the well-known formula for the differential of the exponential function we have
\[
D\phi_x=\sum_{k\ge0}\frac{(-\ad_x)^k}{(k+1)!},
\]
where the sum is of course finite. In other words, if for any $\beta\in\roots_P$ we write
$p_\beta(y)$ for the $\beta$-coordinate of $y$ with respect to the basis $(e_\alpha)_{\alpha\in\roots_P}$,
and for any $y\in\nnn$ we denote by $\varphi_y^\beta$ the smooth function $\varphi_y^\beta(x)=p_\beta(D\phi_x(y))$ on $\nnn$ (which is in fact a polynomial function of the archimedean component), then
\begin{equation} \label{eq: convertder}
\frac{\partial}{\partial y}(f\circ\exp)=\sum_{\beta\in\roots_P}\varphi_y^\beta\cdot(f*e_\beta)\circ\exp.
\end{equation}
Note that for any $\alpha\in\roots_P$ and $x\in\nnn$, $D\phi_x(e_\alpha)$ is contained in the span of $e_\beta, \alpha\prec\beta$.
Thus,
\begin{equation} \label{eq: imdphi}
\varphi_{e_\alpha}^\beta\equiv0\text{ if }\alpha\not\prec\beta.
\end{equation}
Moreover,
\[
\frac{\partial\varphi_y^\beta}{\partial e_\alpha}(x)=
\sum_{k\ge0}\frac{(-1)^k}{(k+1)!}\sum_{0\le j<k}(p_\beta\circ(\ad_x)^{k-j-1}\circ\ad_{e_\alpha}\circ(\ad_x)^j)(y),
\]
and hence
\begin{equation} \label{eq: imdphi2}
\frac{\partial\varphi_y^\beta}{\partial e_\alpha}\equiv0\text{ if }\alpha\not\prec\beta.
\end{equation}

For any sequence $J=(\beta_1,\dots,\beta_m)$, $\beta_i\in\roots_P$, we write $e_J=e_{\beta_1}\cdots e_{\beta_m}\in\univ(\nnn)$.

\begin{lemma} \label{lem: convertliegp}
For any sequence $I=(\alpha_1,\dots,\alpha_l)$, $\alpha_i\in\roots_P$, we can write
\begin{equation} \label{eq: liedervsgrp}
\frac{\partial^l(f\circ\exp)}{\prod_i\partial e_{\alpha_i}}=\sum_J\psi_I^J\cdot (f*e_J)\circ\exp
\end{equation}
for any $f\in C^\infty(N(\A))$, where $J$ ranges over the sequences $(\beta_1,\dots,\beta_m)$, $\beta_i\in\roots_P$, $m\le l$,
such that for every $1\le i\le l$ there exists $1\le j\le m$ such that $\alpha_i\prec\beta_j$, and the $\psi_I^J$ are certain smooth functions on $\nnn$, polynomial in the
archimedean component,
which are independent of $f$.
\end{lemma}

\begin{remark}
The left-hand side of \eqref{eq: liedervsgrp} is unchanged if we permute the $\alpha_i$'s. However, the individual functions
$\psi_I^J$ may depend on the order of the sequence $I$.
\end{remark}

\begin{proof}
Using induction on $m$ and the identity \eqref{eq: convertder} we have the relation \eqref{eq: liedervsgrp}, where
\[
\psi_\emptyset^\emptyset\equiv1,\ \psi_I^\emptyset\equiv0\text{ for }I\ne\emptyset,\
\psi_\emptyset^J\equiv0\text{ for }J\ne\emptyset,
\]
and for $I=(\alpha_1,\dots,\alpha_l)$ and $J=(\beta_1,\dots,\beta_m)$, $l,m>0$, $\psi_I^J$ satisfy the recursion relations
\[
\psi_I^J=\frac{\partial\psi_{(\alpha_1,\dots,\alpha_{l-1})}^J}{\partial e_{\alpha_l}}+
\varphi_{e_{\alpha_l}}^{\beta_m}\cdot\psi_{(\alpha_1,\dots,\alpha_{l-1})}^{(\beta_1,\dots,\beta_{m-1})}.
\]
We first note that using induction on the length of $I$ and property \eqref{eq: imdphi2} we have
\[
\frac{\partial\psi_I^J}{\partial e_\alpha}\equiv0\text{ if }\alpha\not\prec\beta_j\text{ for all }1\le j\le m.
\]
Using this, we can now show by induction on $\abs{I}$ that $\psi_I^J\equiv0$ unless for every $1\le i\le l$
there exists $1\le j\le m$ such that $\alpha_i\prec\beta_j$.
Indeed, for $i<l$ this follows from the induction hypothesis and for $i=l$ this follows from \eqref{eq: imdphi} and the claim above.
\end{proof}

Let now $P_1\subset P_2$ be standard parabolic subgroups.
For brevity we write $\smofrts 12=\smofrts {P_1}{P_2}$, where
$\smofrts {P_1}{P_2}$ has been introduced in \eqref{def: xi}.

\begin{proposition} \label{prop: altsumunipa}
There exist $X_1,\dots,X_m\in\univ(\nnn_1^2)$ such that for any compact open subgroup $K_1$ of $N_1(\A_f)$
and any $a\in A_0$ satisfying $\tau_0^2(\Ht_0(a)-T_1)=1$ we have
\begin{multline*}
\abs{\sum_{P:P_1\subset P\subset P_2}(-1)^{\dim\aaa_P}\sum_{\nu\in N_1^P(\Q)}\int_{N_P(\A)}f(a^{-1}\nu na)\ dn}\ll_{K_1}\\
e^{-\sprod{\smofrts 12}{\Ht_0(a)}}\sum_{i=1}^m\int_{N_1(\A)}\abs{(f*X_i)(a^{-1}na)}\ dn
\end{multline*}
for any $f\in\funct(N_1(\A);K_1)$.
Moreover, if $K_1$ contains $\fund_{P_1,\fin}(\prmlvl)$, then we may take the implied constant to be $\prmlvl^s$,
where $s$ is a positive integer depending only on $G$.
\end{proposition}

\begin{proof}
Upon replacing $f$ by $\int_{N_2(\A)}f(n\cdot)\ dn$ and $G$ by $M_2$, we may assume without loss of generality that $P_2=G$.
We apply Lemma \ref{lem: altsumvs}, taking the coordinates $e_\alpha$, $\alpha\in\roots_{P_1}$.
More precisely, let $I=\roots_{P_1}$, $J=\Delta_0\setminus\Delta_0^1$, and for any $\alpha\in J$ set
$I_\alpha=\roots_{P_\alpha}$, where $P_\alpha$ is the maximal parabolic subgroup of $G$ corresponding to $\alpha$.
Thus, $\cup_{\alpha\notin\Delta_0^P}I_\alpha=\roots_P$ for any $P\supset P_1$.
Take $\prmlvl$ to be a multiple of $\prmlvl_0$ so that $K_1$ contains $\fund_{P_1,\fin}(\prmlvl)$.
By Lemma \ref{lem: altsumvs} we have
\begin{multline*}
\abs{\sum_{P:P_1\subset P}(-1)^{\dim\aaa_P}\int_{\fund_P(\prmlvl)}f(x)\ dx}=
\abs{\sum_{P:P_1\subset P}(-1)^{\dim\aaa_P}\int_{\fund(\prmlvl)^{\roots_P}}f(\exp(\coor_P(\underline{x})))\ d\underline{x}} \le \\
\prmlvl^{\abs{I}} \sideset{}{'}\sum_{I'}\int_{\fund(\prmlvl)^{\roots_{P_1}}}\abs{\frac{\partial^{\abs{I'}}(f\circ\exp\circ\coor_{P_1})}
{\prod_{\alpha\in I'}\partial\underline{x}_\alpha}(\underline{x})}\ d\underline{x}=
\prmlvl^{\abs{I}} \sideset{}{'}\sum_{I'}\int_{\log(\fund_{P_1}(\prmlvl))}\abs{\frac{\partial^{\abs{I'}}(f\circ\exp)}{\prod_{\alpha\in I'}\partial e_\alpha}(x)}\ dx,
\end{multline*}
where the sum is over all $I'\subset I$ such that for any $\alpha\in\Delta_0\setminus\Delta_0^1$ there exists $\beta\in I'$ with $\alpha\prec\beta$.
It follows from Lemma \ref{lem: convertliegp} that
\[
\abs{\sum_{P:P_1\subset P}(-1)^{\dim\aaa_P}\int_{\fund_P(\prmlvl)}f(x)\ dx}\ll
\prmlvl^s \sideset{}{''}\sum_J\int_{\fund_{P_1}(\prmlvl)}\abs{(f*e_J)(x)}\ dx,
\]
for a suitable integer $s$ (depending only on $G$), where $J$ ranges over all sequences $(\beta_1,\dots,\beta_m)$, $m\le\abs{\roots_1}$, such that
for any $\alpha\in\Delta_0\setminus\Delta_0^1$ there exists $1\le j\le m$ with $\alpha\prec\beta_j$.
Summing over all left translates of $f$ by $\nu\in N_1(\Q)$ and using Lemma \ref{lem: fundom} we obtain
\[
\abs{\sum_{P:P_1\subset P}(-1)^{\dim\aaa_P}\sum_{\nu\in N_1^P(\Q)}\int_{N_P(\A)}f(\nu n)\ dn}\ll
\prmlvl^s \sideset{}{''}\sum_J\int_{N_1(\A)}\abs{(f*e_J)(x)}\ dx.
\]
Applying this to $f_a=f(a^{-1}\cdot a)$ we obtain the required bound.
Note that
\[
f_a*e_J(x)=(f*\Ad(a^{-1})e_J)(a^{-1}xa)= \prod_j\beta_j(a^{-1}) (f*e_J)(a^{-1}xa),
\]
and that $\abs{\beta(a^{-1})}\ll1$ for all $\beta\in\roots_P$ by the condition on $a$.
\end{proof}

\begin{corollary} \label{cor: altsumunip}
Let $P_1\subset P_3\subset P_2$ be standard parabolic subgroups.
Then there exist an integer $s$ and $X_1,\dots,X_m\in\univ(\Lieg^1)$ such that
\begin{multline} \label{eq: altsumunip}
\abs{\sum_{P:P_3\subset P\subset P_2}(-1)^{\dim\aaa_P}\sum_{\nu\in N_3^P(\Q)}\int_{N_P(\A)}f(g^{-1}\nu ng)\ dn}\le\\
\level (K)^s
e^{-\sprod{\smofrts 32}T}e^{-\sprod{(\smofrts 32)_{P_1}}{\Ht_0(g)-T}}
\sum_{i=1}^m\int_{N_3(\A)}\abs{(f*X_i)(g^{-1}ng)}\ dn
\end{multline}
for any $f\in\funct(G(\A)^1;K)$ and $g\in G(\A)^1$ such that $F^1(g,T)\tau_1^2(H_{P_1}(g)-T)=1$.
\end{corollary}

\begin{proof}
We follow the argument of \cite{MR518111} and \cite[\S3]{MR828844}.
As a function of $g$, both sides of \eqref{eq: altsumunip} are left $P_1(\Q)N_2(\A)$-invariant.
Hence, we may assume that $g$ is of the form $g=namk$, where $k\in\K$, and $a\in A_0$ satisfies
\begin{equation} \label{eq: siegelP1}
\tau_0^1(\Ht_0(a)-T_1)=1,
\end{equation}
$n$ is in a fixed compact subset of $N_0^2(\A)$ and $m$ is in a fixed compact subset of $M_0(\A)^1$.
As explained in \cite[pp.~943--944]{MR518111}, the condition $F^1(g,T)\tau_1^2(H_{P_1}(g)-T)=1$ implies that
\begin{equation} \label{eq: htp}
\sprod{\alpha}{\Ht_0(a)-T}=\sprod{\alpha_{P_1}}{\Ht_0(a)-T}+\sprod{\alpha-\alpha_{P_1}}{\Ht_0(a)-T}\ge
\sprod{\alpha_{P_1}}{\Ht_0(a)-T}>0
\end{equation}
for all $\alpha\in\Delta_0^2\setminus\Delta_0^1$.
Hence, by \eqref{eq: siegelP1}, $\tau_0^2(\Ht_0(a)-T_1)=1$ and $a^{-1}na$ lies in a fixed compact subset of $N_0^2(\A)$ that is independent of $T$ (and $K$).
Therefore, by Lemma \ref{lem: sob}, part \ref{part: compact conjugation}, we may assume that $g=a$.
This case follows from Proposition \ref{prop: altsumunipa} (with $P_1=P_3$) since by \eqref{eq: htp} we have
\[
\sprod{\smofrts 32}{T}+\sprod{(\smofrts 32)_{P_1}}{\Ht_0(a)-T}\le\sprod{\smofrts 32}{\Ht_0(a)}.
\]
The corollary follows. Note that $K \cap N_3 (\A_f)$ contains
$\fund_{P_3,\fin}(\prmlvl)$ for $\prmlvl = \prmlvl_1 \level (K)$, where
$\prmlvl_1$ is a positive integer depending only on $G$ and the
representation $r_0$ fixed in \S \ref{subsectionlevel}.
\end{proof}

\subsection{}
For the application in the next section, we need another auxiliary result. As in \cite[\S6]{MR518111} let $\sigma_1^2$ be the function
\[
\sigma_1^2=\sum_{P_1\subset P_3\subset P_2}(-1)^{\dim\aaa_3^2}\tau_1^3\hat\tau_3
\]
on $\aaa_0$. This function is described in [ibid., Lemma 6.1 and Corollary 6.2].
In particular, $\sigma_1^2$ is the characteristic function of a certain subset
of $\aaa_0$,
and if $\sigma_1^2(H)=1$ then $\tau_1^2(H)=1$ and $\norm{H}\le c(1+\norm{H_1^2})$, where $H_1^2$ denotes the projection of $H$ to $\aaa_1^2$
and the constant $c$ depends only on $G$.

\begin{lemma} \label{lem: integm1'}
Let $P_1\subset P_3\subset P_2$ be standard parabolic subgroups.
Then for any $X\in\aaa_1^3$ we have
\begin{equation} \label{eq: amp1'}
\int_{\aaa_3}\sigma_1^2(X+X_1-T)e^{-\sprod{(\smofrts 32)_{P_1}}{X+X_1-T}}\ dX_1
\ll(1+\norm{X-T_{P_1}^{P_3}})^{\dim\aaa_2}\tau_1^3(X-T).
\end{equation}
\end{lemma}

\begin{proof}
We write the left-hand side as
\[
\int_{\aaa_3^2}e^{-\sprod{(\smofrts 32)_{P_1}}{X+X_1-T}}\left(\int_{\aaa_2}\sigma_1^2(X+X_1+X_2-T)\ dX_2\right)\ dX_1.
\]
By the above-mentioned property of $\sigma_1^2$, the inner integral is bounded by a constant multiple of
\[
\tau_1^2(X+X_1-T)(1+\norm{X+X_1-T_{P_1}^{P_2}})^{\dim\aaa_2}.
\]
Let $t_\beta=\sprod{\beta}{X+X_1-T}$, $\beta\in\Delta_1^2$. In particular, $t_\beta=\sprod{\beta}{X-T}$ for $\beta\in\Delta_1^3$.
The condition $\tau_1^2(X+X_1-T)=1$ means that $t_\beta>0$ for all $\beta\in\Delta_1^2$, and in particular
implies $\tau_1^3(X-T)=1$. For convenience let $\Delta'=\Delta_1^2\setminus\Delta_1^3$.
Thus, the left-hand side of \eqref{eq: amp1'} is
\begin{multline*}
\ll\tau_1^3(X-T)\int_{\aaa_3^2}\prod_{\beta\in\Delta'}{\bf 1}_{>0}(t_\beta)
e^{-\sum_{\beta\in\Delta'}t_\beta}(1+\sum_{\beta\in\Delta_1^2}\abs{t_\beta})^{\dim\aaa_2}\ dX_1\le\\
(1+\sum_{\beta\in\Delta_1^3}\abs{t_\beta})^{\dim\aaa_2}
\tau_1^3(X-T)\int_{\aaa_3^2}\prod_{\beta\in\Delta'}{\bf 1}_{>0}(t_\beta)
e^{-\sum_{\beta\in\Delta'}t_\beta}(1+\sum_{\beta\in\Delta'}t_\beta)^{\dim\aaa_2}\ dX_1,
\end{multline*}
where ${\bf 1}_{>0}$ is the characteristic function of the positive reals.
The last integral converges since we can replace the integration variable $X_1$ by $t_\beta$, $\beta\in\Delta'$.
The lemma follows.
\end{proof}

\begin{remark} \label{rem: xitxi}
We may obviously replace $(\xi^2_3)_{P_1}$ here by any positive
multiple and obtain the same estimate (changing only the implicit constant).
\end{remark}

\section{Continuity of the coarse geometric expansion} \label{sec: contcoarse}
We now prove the continuity of the coarse geometric expansion
of Arthur's trace formula \cite{MR518111}. We first recall Arthur's derivation of this expansion.
Let for the time being $f\in C_c^\infty(G(\A)^1)$.
For any $T\in\aaa_0$ with $d(T)>d_0$ let $k^T(\cdot,f)$ be Arthur's modified kernel
\begin{equation} \label{eqnkT}
k^T(x,f)=\sum_{P\supset P_0}(-1)^{\dim\aaa_P}\sum_{\delta\in P(\Q)\bs G(\Q)}k_P(\delta x)\hat\tau_P(H_P(\delta x)-T)
\end{equation}
with
\[
k_P(x)=\sum_{\gamma\in M_P(\Q)}\int_{N_P(\A)}f(x^{-1}\gamma nx)\ dn.
\]
The inner sum in \eqref{eqnkT} has only finitely many non-zero terms (and the possible values for $\delta$ depend only on $x$, not on $f$).
Let
\[
J^T(f)=\int_{G(\Q)\bs G(\A)^1}k^T(x)\ dx.
\]
Arthur shows that this integral is absolutely convergent for all $T$ with $d(T)$ large enough, the bound depending on the support of $f$.

Following \cite[\S7]{MR518111}, we can invoke Arthur's partition lemma \eqref{eq: Arpar} to rewrite \eqref{eqnkT} as
\[
k^T(x)=\sum_{P_1\subset P_2}\sum_{\delta\in P_1(\Q)\bs G(\Q)}F^1(\delta x,T)\sigma_1^2(H_{P_1}(\delta x)-T)k_{1,2}(\delta x),
\]
where
\[
k_{1,2}(x)=\sum_{P:P_1\subset P\subset P_2}(-1)^{\dim\aaa_P}k_P(x).
\]


For $\gamma_1,\gamma_2\in G(\Q)$ we write $\gamma_1\simw\gamma_2$ if the semisimple parts of $\gamma_1$ and $\gamma_2$ are conjugate
in $G(\Q)$. The equivalence classes of $\simw$ are called coarse classes.
Let $\cOrb$ be the set of all coarse classes.
Each $\corb\in\cOrb$ contains a unique semisimple conjugacy class of $G(\Q)$.
For any $\corb\in\cOrb$ Arthur sets
\[
k_{\corb}(x)=\sum_{\gamma\in\corb}f(x^{-1}\gamma x)
\]
and
\[
k^T_{\corb}(x)=\sum_{P\supset P_0}(-1)^{\dim\aaa_P}\sum_{\delta\in P(\Q)\bs G(\Q)}k_{\corb,P}(\delta x)\hat\tau_P(H_P(\delta x)-T)
\]
with
\[
k_{\corb,P}(x)=\sum_{\gamma\in M_P(\Q)\cap\corb}\int_{N_P(\A)}f(x^{-1}\gamma nx)\ dn.
\]
A basic fact \cite[p. 923]{MR518111} is that
\begin{equation} \label{eq: lp}
\corb\cap P(\Q)=(\corb\cap M(\Q))N(\Q).
\end{equation}

Clearly, we have
\[
k^T(x)=\sum_{\corb\in\cOrb}k^T_{\corb}(x).
\]
The integrals
\[
J^T_\corb(f)=\int_{G(\Q)\bs G(\A)^1}k^T_\corb(x)\ dx
\]
are again absolutely convergent for $d(T)$ large enough (depending on the support of $f$) and
we have the decomposition $J^T (f) = \sum_{\corb\in\cOrb}J^T_{\corb}(f)$ for all such $T$.

Exactly as before, we can write
\begin{equation} \label{eq: expP12}
k^T_{\corb}(x)=\sum_{P_1\subset P_2}\sum_{\delta\in P_1(\Q)\bs G(\Q)}F^1(\delta x,T)\sigma_1^2(H_{P_1}(\delta x)-T)k_{\corb,1,2}(\delta x),
\end{equation}
where
\[
k_{\corb,1,2}(x)=\sum_{P:P_1\subset P\subset P_2}(-1)^{\dim\aaa_P}k_{\corb,P}(x).
\]

We now extend Arthur's convergence results as follows.

\begin{theorem} \label{thm: maincoarse}
\begin{enumerate}
\item For any $f\in\funct(G(\A)^1;K)$, $\corb\in\cOrb$ and $T
\in\aaa_0$ such that $d(T)>d_0$, the integrals defining $J^T(f)$ and $J^T_{\corb}(f)$ are absolutely convergent.
\item \label{part: polynomial}
$J^T(f)$ and $J^T_{\corb}(f)$ are polynomials in $T$ of degree $\le\dim\aaa_0$ whose coefficients are continuous linear forms in $f$.
\item \label{part: estimcoarse}
There exist $r\ge0$ and a continuous seminorm $\mu$ on $\funct(G(\A)^1;K)$ such that
\begin{multline*}
\sum_{\corb\in\cOrb}\abs{\int_{G(\Q)\bs G(\A)^1_{\le T}}k_{\corb}(x)\ dx-J^T_{\corb}(f)}\le
\sum_{\corb\in\cOrb}\int_{G(\Q)\bs G(\A)^1}\abs{F(x,T)k_{\corb}(x)-k^T_{\corb}(x)}\ dx\\\le
\mu(f)(1+\norm{T})^re^{-d(T)}
\end{multline*}
for any $f\in\funct(G(\A)^1;K)$ and $T$ such that $d(T)>d_0$.
\item $J^T(f)=\sum_{\corb\in\cOrb}J^T_{\corb}(f)$.
\end{enumerate}
In addition, the absolute values of the coefficients in part \ref{part: polynomial} and the
seminorm $\mu$ of part \ref{part: estimcoarse} can be bounded by
$c \level (K)^s \norm{\cdot}_{G(\A)^1,\, t}$ for
a constant $c$ and
positive integers $s, t$ that do not depend on $K$.
\end{theorem}


\begin{proof}
First note that each $k_P(x)$, and hence the modified kernel $k^T(f)$, is well-defined for any $f\in\funct(G(\A)^1;K)$ by Lemma \ref{lem: sob}.
Also, by \eqref{eq: P=G} there exist $r\ge0$ and a continuous seminorm $\mu$ on $\funct(G(\A)^1;K)$ such that
\[
\sum_{\corb\in\cOrb}\int_{G(\Q)\bs G(\A)^1_{\le T}}\abs{k_{\corb}(x)}\ dx\le\mu(f)(1+\norm{T})^r
\]
for any $f\in\funct(G(\A)^1;K)$ and $T$ such that $d(T)>d_0$.
It follows that part \ref{part: estimcoarse} implies the convergence of $J^T(f)$ and $J^T_{\corb}(f)$, as well as the
relation $J^T(f)=\sum_{\corb\in\cOrb}J^T_{\corb}(f)$.
Moreover, Arthur's argument in \cite[\S2]{MR625344} shows that part \ref{part: polynomial} is valid for any $f$ for which $J^T(f)$
(or $J^T_{\corb}(f)$) is absolutely convergent.

Thus, it remains to prove part \ref{part: estimcoarse}.
Following \cite[\S7]{MR518111} (and recalling \eqref{eq: lp}) we can write
\[
k_{\corb,1,2}(x)=\sum_{P_1\subset P_3\subset P_2}k_{\corb,1,2;3}(x),
\]
where
\[
k_{\corb,1,2;3}(x)=\sum_{\eta\in M_3(\Q)\rg{1}\cap\corb}\sum_{P:P_3\subset P\subset P_2}(-1)^{\dim\aaa_P}
\sum_{\nu\in N_3^P(\Q)}\int_{N_P(\A)}f(x^{-1}\eta\nu nx)\ dn
\]
with the definition
\[
M_3(\Q)\rg{1}=M_3(\Q)\setminus\cup_{P:P_1\subset P\subsetneq P_3}P(\Q).
\]

As a side remark, we note that Arthur shows in [ibid., pp. 943--944], that for compactly supported $f$ only the terms with $P_3=P_1$ contribute, provided that $T$ is large with respect to the support of $f$.

The function $k_{\corb,1,2;3}$ is left $P_1(\Q)N_2(\A)$-invariant, since the set $M_3(\Q)\rg{1}$ is invariant under conjugation
by $P_1(\Q)$.
Recall the decomposition \eqref{eq: expP12} of $k_\corb^T (x)$,
and observe that $\sigma_P^P$ vanishes identically unless $P=G$ and that the contribution from $P_1=P_2=G$ to
\eqref{eq: expP12} is simply $F(x,T)k_{\corb}(x)$.
In order to prove part \ref{part: estimcoarse}, it therefore
suffices to
show that there exist $r\ge0$ and a continuous seminorm $\mu$ on $\funct(G(\A)^1;K)$ such that for any triplet
$P_1\subset P_3\subset P_2$ with $P_1\neq P_2$ we have
\[
\sum_{\corb}\int_{P_1(\Q)\bs G(\A)^1}F^1(x,T)\sigma_1^2(H_{P_1}(x)-T)\abs{k_{\corb,1,2;3}(x)}\ dx
\le \mu(f)(1+\norm{T})^re^{-\sprod{\smofrts 12}T}.
\]
By Corollary \ref{cor: altsumunip} (applied to suitable left translates of $f$), we have
\[
\abs{k_{\corb,1,2;3}(x)}\le
e^{-\sprod{\smofrts 32}T}e^{-\sprod{(\smofrts 32)_{P_1}}{\Ht_0(x)-T}}
\sum_i\sum_{\eta\in M_3(\Q)\rg{1}\cap\corb}\int_{N_3(\A)}\abs{(f*X_i) (x^{-1}\eta nx)}\ dn,
\]
provided that $F^1(x,T)\sigma_1^2(H_{P_1}(x)-T)=1$.
Since $\smofrts 12 = \smofrts 13 + \smofrts 32$, it remains to show that
\begin{multline*}
\int_{P_1(\Q)\bs G(\A)^1}F^1(x,T)\sigma_1^2(H_{P_1}(x)-T)e^{-\sprod{(\smofrts 32)_{P_1}}{\Ht_0(x)-T}}
\sum_{\eta\in M_3(\Q)\rg{1}}\int_{N_3(\A)}\abs{f(x^{-1}\eta nx)}\ dn\ dx\\
\le \mu(f)(1+\norm{T})^re^{-\sprod{\smofrts 13}T}
\end{multline*}
with a suitable continuous seminorm $\mu$.
At this point we note that using Lemma \ref{lem: sob}, part \ref{part: pos}, we may assume without loss of generality that $f\ge0$ and $K=\K_{\fin}$.
Using the Iwasawa decomposition with respect to $P_3$, we need to estimate
\[
\int_{P_1^3(\Q)\bs M_3(\A)\cap G(\A)^1}e^{-\sprod{(\smofrts 32)_{P_1}}{\Ht_0(x)-T}}F^1(x,T)\sigma_1^2(H_{P_1}(x)-T)
\sum_{\eta\in M_3(\Q)\rg{1}}f_{P_3}(x^{-1}\eta x)\ dx,
\]
where $P_1^3=P_1\cap M_3$, and $f_{P_3}$ is as in
\eqref{eq: f_P}.
Splitting $M_3(\A)\cap G(\A)^1$ as the direct product of $A_{M_3}$ and $M_3(\A)^1$, we may
estimate the integral over $A_{M_3}$ using Lemma \ref{lem: integm1'}. The above integral is then majorized by
a constant multiple of
\[
\int_{P_1^3(\Q)\bs M_3(\A)^1}F^1(x,T)(1+\norm{\Ht_{P_1}(x)-T_{P_1}^{P_3}})^{\dim\aaa_2}\tau_1^3(H_{P_1}(x)-T)
\sum_{\eta\in M_3(\Q)\rg{1}}f_{P_3}(x^{-1}\eta x)\ dx.
\]
We can now appeal to Theorem \ref{thm: analog} (with $G=M_3$ and $P=P_1^3$) and Lemma \ref{lem: sob}, part \ref{part: consterm}, to finish the proof.
\end{proof}

\section{Modifications for the finer classes} \label{sec: modif}
In this section we will fine-tune the results of \S\ref{sec: auxil} to adapt the continuity argument to the decomposition of the trace formula
with respect to the equivalence relation $\sim$ of the introduction instead of $\sim_w$. The goal is to prove Corollary
\ref{cor: altsumorb} below, which is the main technical prerequisite for the continuity argument in \S\ref{sec: main}.

\subsection{}
We first go back to the situation of Lemmas \ref{lem: altsumvscls} and \ref{lem: altsumvs}.
Namely, $I$ is a finite set and $I_j$, $j\in J$, is a family of (not necessarily disjoint) non-empty subsets of $I$.
We also have an integer parameter $\prmlvl\ge1$.
Recall that for any $J'\subset J$ we denote by $[0,\prmlvl]^I_{J'}$
the face of $[0,\prmlvl]^I$
consisting of the vectors whose support is contained in $\cup_{j\notin J'}I_j$.

We say that a family $\FFF$ of subsets of $J$ is \emph{monotone} if
whenever $J'\in\FFF$ and $J'\subset J''\subset J$, we also have $J''\in\FFF$.

\begin{lemma} \label{lem: altsumvs2}
Let $\FFF$ be a monotone family of subsets of $J$ and
let $\FFF_{\min}$ be the set of minimal elements of $\FFF$ with respect to inclusion.
Then
there exist coefficients $d_{I'} \in \Z$, depending only on
$\FFF$, such that
for any $f\in C^{\abs{I}}([0,\prmlvl]^I)$ we have
\begin{equation} \label{eq: sumlac}
\sum_{J'\in\FFF}(-1)^{\abs{J\setminus J'}}\int_{[0,\prmlvl]^I_{J'}}f(x)\ dx=
\sideset{}{''}\sum_{I'}d_{I'}\int_{[0,\prmlvl]^I}\frac{\partial^{\abs{I'}}f}{\prod_{i\in I'}\partial x_i}(x)\prod_{i\in I'}(x_i-\prmlvl)\ dx,
\end{equation}
where the sum is over all $I'\subset I$ such that
$I'\cap I_j\ne\emptyset$ for every $j\in J\setminus\cup\FFF_{\min}$.
\end{lemma}

\begin{proof}
This is a special case of \eqref{eq: gencoefip} with
\[
c_{I'}=\sum_{J'\in\FFF:\cup_{j\notin J'}I_j=I\setminus I'}(-1)^{\abs{J_0\setminus J'}}.
\]
Note that
\[
d_{I'}=\sum_{I''\supset I'}c_{I''}=\sum_{J'\in\FFF:\cup_{j\notin J'}I_j\subset I\setminus I'}(-1)^{\abs{J_0\setminus J'}}=
\sum_{\{j\in J:I_j\cap I'\ne\emptyset\}\subset J'\in\FFF}(-1)^{\abs{J_0\setminus J'}}.
\]
Observe that if $j_0\in J\setminus\cup\FFF_{\min}$ then for any $J'\subset J$,
$J'\in\FFF$ if and only if $J'\cup\{j_0\}\in\FFF$.
Similarly, if $I'\cap I_{j_0}=\emptyset$ then for any $J'\subset J$,
$\{j\in J:I_j\cap I'\ne\emptyset\}\subset J'$ if and only if
$\{j\in J:I_j\cap I'\ne\emptyset\}\subset J'\cup\{j_0\}$.
Thus, $d_{I'}=0$ if there exists $j_0\in J\setminus\cup\FFF_{\min}$ such that $I'\cap I_{j_0}=\emptyset$,
as required. 
\end{proof}

Once again, an adelic version follows immediately.
\begin{corollary} \label{cor: altsumvs23}
Let $I$, $I_j$, $j \in J$, and $\FFF$ be as in
Lemma \ref{lem: altsumvs2} above.
Then we have (using the notation of Lemma \ref{lem: altsumvs})
\[
\abs{\sum_{J'\in\FFF}(-1)^{\abs{J'}}\int_{\fund(\prmlvl)^I_{J'}}f(x)\ dx}\ll_{\FFF}
\prmlvl^{\abs{I}}
\sideset{}{''}\sum_{I'}\int_{\fund(\prmlvl)^I}\abs{\frac{\partial^{\abs{I'}}f}{\prod_{i\in I'}\partial x_i}(x)}\ dx
\]
for any $f\in C^{\abs{I}}(\A^I;\fund_{\fin}(\prmlvl)^I)$.
\end{corollary}

We say that a non-empty family $\pars$ of parabolic subgroups of $G$ is \emph{monotone}
if whenever $Q\in\pars$ and $Q'\supset Q$, we also have $Q'\in\pars$. For a monotone family $\pars$
let $\pars_{\min}$ be the set of minimal elements of $\pars$ (with respect to inclusion)
and let $Q(\pars)$ be the parabolic subgroup generated by the elements of $\pars_{\min}$.

\begin{corollary} \label{cor: altsumunipa2}
Let $P$ be a standard parabolic subgroup.
Let $\pars$ be a monotone family of parabolic subgroups of $G$ which all contain $P$.
Then there exist an integer $s$ and $X_1,\dots,X_m\in\univ(\nnn)$
such that for any $\prmlvl$ divisible by $\prmlvl_0$ and $a\in A_0$ satisfying $\tau_0(\Ht_0(a)-T_1)=1$ we have
\begin{multline*}
\sum_{\nu\in N(\Q)}\abs{\sum_{Q\in\pars}(-1)^{\dim\aaa_Q}\int_{\fund_Q(\prmlvl)}f(a^{-1}\nu na)\ dn}\\
\le {\prmlvl}^s e^{-\sprod{\smofrts{Q(\pars)}{}}{\Ht_0(a)}}\sum_{i=1}^m\int_{N(\A)}\abs{(f*X_i)(a^{-1}na)}\ dn
\end{multline*}
for every $f\in\funct(N(\A);\fund_{P,\fin}(\prmlvl))$.
\end{corollary}

\begin{proof}
This follows from Corollary \ref{cor: altsumvs23} exactly as in the proof of Proposition \ref{prop: altsumunipa}.
Recall that $I=\roots_P$, $J=\Delta_0\setminus\Delta_0^P$ and that
$I_\alpha=\roots_{P_\alpha}$ for $\alpha\in J$,
where $P_\alpha$ is the maximal parabolic subgroup of $G$ corresponding to $\alpha$.
In the case at hand we have $\FFF=\{\Delta_0^Q\setminus\Delta_0^P:Q\in\pars\}$ and $\FFF_{\min}=\{\Delta_0^Q\setminus\Delta_0^P:Q\in\pars_{\min}\}$, and therefore
$\Delta_0^{Q(\pars)}\setminus\Delta_0^P=\cup\FFF_{\min}$.
\end{proof}

\subsection{}
The following is a version of \cite[Proposition 5.3.1]{MR3427596}.

\begin{lemma} \label{lem: diagt}
Let $a_1,\dots,a_n\in\R^*$ and let $\alpha:\A^n\rightarrow\A^n$ be the $\A$-linear transformation given by
$\alpha (x_1,\dots,x_n)=(a_1x_1,\dots,a_nx_n)$. Fix a constant $c>0$ and assume that $\abs{a_j}\le c$ for all $j$.
Then for a non-zero polynomial $\phi\in\overline{\Q}[x_1,\dots,x_n]$
we have
\begin{equation} \label{eq: bndphi=0}
\sum_{v\in \Q^n:\phi(v)=0}\abs{f(\alpha(v))}\ll_{c,n} \prmlvl^n \deg\phi\max\abs{a_j}\sum_{I\subset\{1,\dots,n\}}
\int_{\A^n}\abs{\frac{\partial^{\abs{I}}f}{\prod_{i\in I}\partial x_i}(\alpha(v))}\ dv
\end{equation}
for any $f\in\funct(\A^n;\fund_{\fin}(\prmlvl)^n)$.
\end{lemma}

\begin{proof}
Note first that the bound
\begin{equation} \label{eq: trivbnd}
\sum_{v\in \Q^n}\abs{f(\alpha(v))}\ll_{c,n} \prmlvl^n \sum_{I\subset\{1,\dots,n\}}
\int_{\A^n}\abs{\frac{\partial^{\abs{I}}f}{\prod_{i\in I}\partial x_i}(\alpha(v))}\ dv
\end{equation}
follows already from Lemma \ref{lem: sob}, part \ref{part: ltcsum}.

Let $d=\deg\phi$.
We prove the lemma by induction on $n$. In the case $n=1$, \eqref{eq: bndphi=0} follows from the inequality
\[
\sup\abs{f}\le \prmlvl (\norm{f}_{L^1(\A)}+\norm{f'}_{L^1(\A)}),
\]
and the fact that the number of roots of $\phi$ is at most $d$.

For the induction step, we write
\[
\phi(x_1,\dots,x_n)=\sum_{i=0}^d\phi_i(x_1,\dots,x_{n-1})x_n^i,
\]
where at least one of the polynomials $\phi_i$, say $\phi_{i_0}$, is non-zero.
Let $p:\Q^n\rightarrow \Q^{n-1}$ be the projection to the first $n-1$ coordinates.
We split the sum on the left-hand side of \eqref{eq: bndphi=0} into two according to whether or not $\phi_{i_0}(p(v))=0$.
The first sum is bounded by
\[
\sum_{v'\in \Q^{n-1}:\phi_{i_0}(v')=0}\sum_{x_n\in \Q}\abs{f(\alpha(v',x_n))}
\]
which is majorized by
\[
\prmlvl \sum_{v'\in \Q^{n-1}:\phi_{i_0}(v')=0}\int_{\A}\big(\abs{f(\alpha(v',x_n))}+\abs{\frac{\partial f}{\partial x_n}(\alpha(v',x_n))}\big)\ dx_n,
\]
to which we can apply the induction hypothesis. The second sum is
\[
\sum_{v'\in \Q^{n-1}:\phi_{i_0}(v')\ne0}\sum_{x_n\in \Q:\phi(v',x_n)=0}\abs{f(\alpha(v',x_n))}.
\]
By the case $n=1$, the inner sum is bounded by
\[
\ll_{c} \prmlvl d\abs{a_n}\int_{\A}\big(\abs{f(\alpha(v',x_n))}+\abs{\frac{\partial f}{\partial x_n}(\alpha(v',x_n))}\big)\ dx_n,
\]
since $\phi(v',x_n)$ is a non-zero polynomial in $x_n$ of degree $\le d$ if $\phi_{i_0}(v')\ne0$.
We can now use \eqref{eq: trivbnd} for the first $n-1$ variables to bound the sum over $v'$,
no longer using the condition $\phi_{i_0}(v')\ne0$.
\end{proof}

\begin{corollary} \label{cor: savephi}
Let $P$ be a maximal parabolic subgroup of $G$ with $\Delta_0\setminus\Delta_0^P=\{\alpha\}$. Then there exist $X_1,\dots,X_m\in\univ(\nnn)$ such that
for any non-zero polynomial $\phi$ on $\nnn$ of degree $\le d$ and any compact open subgroup $K'$ of $N(\A_f)$ we have
\[
\sum_{n\in N(\Q):\phi(\log n)=0}\abs{f(a^{-1}na)}\ll_{K'} d
\, e^{-\sprod{\alpha}{\Ht_0(a)}}\sum_{i=1}^m\int_{N(\A)}\abs{(f*X_i)(a^{-1}na)}\ dn
\]
for any $f\in\funct(N(\A);K')$ and $a\in A_0$ such that $\tau_0(\Ht_0(a)-T_1)=1$.
Moreover, if $K'$ contains $\fund_{P,\fin}(\prmlvl)$, then we may take the implied constant to be $\prmlvl^s$, where $s$ is a positive integer depending only on $G$.
\end{corollary}

\begin{proof}
This immediately follows from Lemma \ref{lem: diagt} and Lemma \ref{lem: convertliegp}.
Note that for any $\beta\in\roots_P$ we have $\sprod{\beta}{\Ht_0(a)}\ge\sprod{\alpha}{\Ht_0(a)}-C$
for some constant $C$ (depending on $T_1$).
\end{proof}


\subsection{}
We recollect some known facts about the ``parabolic induction'' of conjugacy classes \`a la Lusztig-Spaltenstein \cite{MR527733}.
In the current setup this notion was considered by Hoffmann \cite{1206.3068}.

Let $P=M\ltimes N$ be a parabolic subgroup of $G$ (defined over $\Q$) and let $\gamma\in M$.
As explained in \cite{1206.3068}, there exists a unique conjugacy class $\geom{I}_P(\gamma)=\geom{I}_P^G(\gamma)$ of $G$ which intersects
$\gamma N$ in a Zariski open dense set.
This follows from the fact that there are only finitely many conjugacy classes of $G$ intersecting $\gamma N$
and that each conjugacy class is a locally closed subvariety.
If $\gamma\in M(\Q)$ then $\geom{I}_P(\gamma)$ is defined over $\Q$ and
we let $I_P(\gamma)=\geom I_P(\gamma)\cap\corb_\gamma$ where $\corb_\gamma$ is the coarse class of $\gamma$ in $G(\Q)$.
Note that $I_P(\gamma)$ is non-empty (and consists of a union of conjugacy classes of $G(\Q)$) since
it contains the rational points of a dense Zariski open subset of $\gamma N$.
(We recall that $\gamma N(\Q)\subset\corb_\gamma$.)
More generally, if $Q\supset P$ we will write $\geom I_P^Q(\gamma)=\geom I_{P\cap M_Q}^{M_Q}(\gamma)$
and similarly for $I_P^Q(\gamma)$ (if $\gamma\in M_P(\Q)$).
It will also be convenient to set $\geom I_P^Q(\gamma)=\geom I_P^Q(\gamma_M)$ for any $\gamma\in P$, where
$\gamma_M$ is the projection of $\gamma$ to $M$, and to define $I_P^Q(\gamma)$ similarly if $\gamma\in P(\Q)$.

Let $\gamma\in P$ and suppose that $Q$ is a parabolic subgroup of $G$ containing $P$.
Induction is transitive in the sense that
if $Q\supset P$ and $\eta\in\geom I_P^Q(\gamma)$ then $\geom I_Q(\eta)=\geom I_P(\gamma)$
(see \cite{MR527733}).
Thus,
\begin{equation} \label{eq: transitive}
\text{if $\gamma\in P(\Q)$ and $\eta\in I_P^Q(\gamma)$ then $I_Q(\eta)=I_P(\gamma)$.}
\end{equation}
Another simple property is that if $\delta\in\overline{[\gamma]_M}$, where $[\gamma]_M$ is the conjugacy class of $\gamma$ in $M$, then
\begin{equation} \label{eq: monotone}
\geom I_P(\delta)\subset\overline{\geom I_P(\gamma)}.
\end{equation}
Indeed, $\overline{\geom I_P(\gamma)}\supset \overline{[\gamma]_MN}\supset\overline{[\gamma]_M}N\supset\delta N$
and $\geom I_P(\delta)\cap\delta N\ne\emptyset$.
It follows from the two properties above that
\begin{equation} \label{eq: monind}
\geom I_Q(\gamma)\subset\overline{\geom I_P(\gamma)}
\text{ for } Q \supset P,
\end{equation}
since $\gamma\in\overline{\geom I_P^Q(\gamma)}$.

We will need an additional qualitative property pertaining to induced classes.
By the definition of $\geom I_P(\gamma)$, for each $\gamma \in M$ there exists a non-zero regular function $F_\gamma$ on $N$ which vanishes
on the complement of $\gamma^{-1}\geom I_P(\gamma)$ in $N$.
\begin{lemma} \label{lem: bndeg}
We may choose $F_\gamma$ so that the degree
of the polynomial function $F_\gamma \circ \exp$
on $\nnn$ is bounded in terms of $G$ only.
\end{lemma}

Let $\spart{\gamma}$ (resp., $\upart{\gamma}$) be the semisimple (resp., unipotent) part of $\gamma$ in the Jordan--Chevalley decomposition.
Denote by $N_{\spart{\gamma}}$ the centralizer of $\spart{\gamma}$ in $N$.
In order to prove Lemma \ref{lem: bndeg}, we use the following result of Arthur on algebraic groups. (We have already used it implicitly when quoting relation \eqref{eq: lp}.)
\begin{lemma} (\cite[Lemma 2.1]{MR518111}) \label{lem: arthur}
For any $n\in N$ there exists $u\in N$, unique up to left translation by $N_{\spart{\gamma}}$,
such that $u\gamma nu^{-1}\in\gamma N_{\spart{\gamma}}$.
The map $n\mapsto u$ defines a morphism $\quotmap_\gamma:N\rightarrow N_{\spart{\gamma}}\bs N$ of affine varieties.
Moreover, if $\gamma\in M(\Q)$ then $\quotmap_\gamma$ is defined over $\Q$.
\end{lemma}

We will need to know that in a suitable sense $\quotmap_\gamma$ is algebraic in $\gamma$.
Before making this more precise we prove another lemma.
Let $G_{\spart{\gamma}}$ (resp., $P_{\spart{\gamma}}$, $M_{\spart{\gamma}}$) be the connected component of the identity
of the centralizer of $\spart{\gamma}$ in $G$ (resp., $P$, $M$).
It is well known that $P_{\spart{\gamma}}$ is a parabolic subgroup of $G_{\spart{\gamma}}$ with Levi decomposition $P_{\spart{\gamma}}=M_{\spart{\gamma}}\ltimes N_{\spart{\gamma}}$.
\begin{lemma} \label{lem: wc}
Let $n\in N$ and suppose that there exists $u\in N$ such that
$\gamma^{-1}u\gamma nu^{-1}\in N_{\spart{\gamma}}\cap\upart{\gamma}^{-1}\geom I_{P_{\spart{\gamma}}}^{G_{\spart{\gamma}}}(\upart{\gamma})$.
Then $n\in\gamma^{-1}\geom I_P(\gamma)$.
\end{lemma}

\begin{proof}
The set $X_\gamma:=\gamma^{-1} \geom I_P(\gamma)\cap N_{\spart{\gamma}}$ is non-empty and Zariski open in $N_{\spart{\gamma}}$, since if
$\gamma n\in \geom I_P(\gamma)$ then by Lemma \ref{lem: arthur} there exists $u\in N$ such that $u^{-1}\gamma nu\in \gamma N_{\spart{\gamma}}$.
Therefore $X_\gamma$ intersects non-trivially the Zariski open and dense set
$Y_\gamma:=\upart{\gamma}^{-1} \geom I_{P_{\spart{\gamma}}}^{G_{\spart{\gamma}}}(\upart{\gamma})\cap N_{\spart{\gamma}}$
of $N_{\spart{\gamma}}$. Let $z\in X_\gamma\cap Y_\gamma$.
Then for any $y\in Y_\gamma$ there exists $g\in G_{\spart{\gamma}}$ such that $\upart{\gamma}y=g^{-1}\upart{\gamma}xg$.
Thus $\gamma y=g^{-1}\gamma zg$ and hence $y\in X_\gamma$. It follows that $Y_\gamma\subset X_\gamma$, hence the lemma.
\end{proof}

\begin{proof}[Proof of Lemma \ref{lem: bndeg}]
Fix a Borel subgroup $B\subset P$ of $G$ (so that its unipotent radical $U$ contains $N$) and a maximal torus $S$ of $G$ contained in $B\cap M$.
(Of course $S$, $B$ and $U$ are not necessarily defined over $\Q$.)
Let $R(S,N)$ be the set of roots of $S$ in $N$.
For any subset $I$ of $R(S,N)$ let $X_I$ be the affine subvariety of $B\cap M$ consisting of the elements $\gamma$ such that $\spart{\gamma}\in S$
and $I=\{\alpha\in R(S,N):\alpha(\spart{\gamma})=1\}$. Also, let $N_I$ be the subgroup of $N$ generated by the root subgroups corresponding
to the roots in $I$. Then $N_{\spart{\gamma}}=N_I$ if $\gamma\in X_I$.
The proof of \cite[Lemma 2.1]{MR518111} shows that $(\gamma,n) \mapsto \quotmap_\gamma(n)$ defines a regular map
\[
\quotmap_I:X_I\times N\rightarrow N_I\bs N.
\]
For completeness we provide the details, since the setup of \cite[Lemma 2.1]{MR518111} is slightly different.
Let $N=U_0\supset\dots\supset U_r=0$ be a sequence of subgroups of $N$ normalized by $B$, such that for all $i=0,\dots,r-1$,
$U_i/U_{i+1}\simeq\affine$ and $[U_i,U]\subset U_{i+1}$.
We claim that for every $i$ there exists a morphism $\quotmap_i:X_I\times N\rightarrow U_iN_I\bs N$
of affine varieties characterized by the property that
\[
\gamma^{-1}\quotmap_i(\gamma,n)\gamma n\quotmap_i(\gamma,n)^{-1}\in U_iN_I\text{ for any }\gamma\in X_I, n\in N.
\]
The case $i=r$ is then the sought-after result.
We argue by induction on $i$. The assertion is trivial for $i=0$. For the induction step, assume that $\quotmap_i$ is defined for some $i<r$.
Let $\alpha_i$ be the root corresponding to $U_i/U_{i+1}$.
If $\alpha_i\in I$ then $U_iN_I=U_{i+1}N_I$ and we simply take $\quotmap_{i+1}=\quotmap_i$.
Otherwise, fixing $u\in N$ such that $x:=\gamma^{-1}u\gamma nu^{-1}\in U_iN_I$ we need to show that there exists $v\in U_i$,
uniquely determined modulo $U_{i+1}$, such that $y_v:=\gamma^{-1}vu\gamma nu^{-1}v^{-1}\in U_{i+1}N_I$.
Note that
\[
y_v=\gamma^{-1}v\gamma xv^{-1}=\gamma^{-1}v\gamma v^{-1}x[x^{-1},v]=\spart{\gamma}^{-1}v\spart{\gamma}
\cdot \spart{\gamma}^{-1}[v^{-1},\upart{\gamma}^{-1}]\spart{\gamma}\cdot v^{-1}x[x^{-1},v]
\]
so that
\[
y_v\in [\spart{\gamma}^{-1},v]xU_{i+1}.
\]
The map $v\mapsto [\spart{\gamma}^{-1},v]$ induces an isomorphism of $U_i/U_{i+1}$ which under the identification with $\affine$
is given by multiplication by $\alpha_i(\spart{\gamma})-1$. Thus $q_{i+1}$ is defined and it is clearly a morphism
(by choosing an algebraic section $U_iN_I\bs N\rightarrow N$ for the quotient map). This finishes the construction of $q_I$.

Fix an algebraic section $\sctn_I:N_I\bs N\rightarrow N$ for the canonical projection $N\rightarrow N_I\bs N$ such that $\sctn_I(N_I)=e$.
Let
\[
\ngamma_I:X_I\times N\rightarrow N_I
\]
 be the regular map given by
\[
\ngamma_I(\gamma,n)=\gamma^{-1}\sctn_I(\quotmap_I(\gamma,n))\gamma n\sctn_I(\quotmap_I(\gamma,n))^{-1}.
\]
For any $\gamma\in X_I$, the map $\ngamma_I(\gamma,\cdot):N\rightarrow N_I$ is surjective,
since its restriction to $N_I$ is the identity map.

We can now conclude Lemma \ref{lem: bndeg}.
For any subset $J$ of $R(S,U)$ let $Y_J$ be the subvariety of $B\cap M$ consisting of elements $\gamma$ such that $\spart{\gamma}\in S$
and $J=\{\alpha\in R(S,U):\alpha(\spart{\gamma})=1\}$.
Thus, $Y_J\subset X_I$ for $I=J\cap R(S,N)$.
Upon conjugating $\gamma$ by an element of $M$ we may assume without loss of generality that $\gamma\in Y_J$ for some $J$.
Let $G_J$ (resp., $P_J$; $M_J$) be the subgroup of $G$ (resp., $P$; $M$) generated by $S$ and the root subgroups corresponding to $J$ and $-J$
(resp., $J$ and $-(J\cap R(S,U\cap M))$; $\pm(J\cap R(S,U\cap M))$).
Thus, $G_J=G_{\spart{\gamma}}$, $P_J=P_{\spart{\gamma}}$, $M_J=M_{\spart{\gamma}}$ and $P_J$ is a parabolic subgroup of $G_J$ with Levi decomposition
$P_J=M_J\ltimes N_I$. (Recall that $N_I=N_{\spart{\gamma}}$.)
Since there are only finitely many unipotent conjugacy classes in $M_I$, we may regard $\upart{\gamma}$ (as well as $J$) as fixed.
By Lemma \ref{lem: wc}, $\gamma n\in\geom I_P(\gamma)$ if for some $u\in N$ we have
$\gamma^{-1}u\gamma nu^{-1}\in N_I\cap\upart{\gamma}^{-1}\geom I_{P_J}^{G_J}(\upart{\gamma})$.
Thus, $n\in\gamma^{-1}\geom I_P(\gamma)$ if $\ngamma_I(\gamma,n)\in  (\upart{\gamma})^{-1}\geom I_{P_J}^{G_J}(\upart{\gamma})$.
Hence, if we choose a non-zero regular function $f$ on $N_I$ which vanishes on the complement of
$(\upart{\gamma})^{-1}\geom I_{P_J}^{G_J}(\upart{\gamma})$ in $N_I$ then we can take
$F_\gamma=f\circ\ngamma_I(\gamma,\cdot)$. Lemma \ref{lem: bndeg} follows.
\end{proof}

\subsection{}
The following result is modeled after \cite[\S6]{MR3427596}.
For standard parabolic subgroups $P_1\subset P_2$ of $G$ define
\[
\asmofrts 12=\begin{cases}\frac{\smofrts 12}{\dim\aaa_1^2}, &P_1\subsetneq P_2,\\0, &\text{otherwise.}\end{cases}
\]
\begin{proposition} \label{prop: subsetY}
Let $P_1\subset P_2$ be standard parabolic subgroups of $G$.
There exist $X_1,\dots,X_m\in\univ(\nnn_1^2)$, such that for a closed subvariety $V$ of $G$, a compact open subgroup
$K_1$ of $N_1(\A_f)$ and an element $\eta\in M_1(\Q)$ we have
\begin{multline} \label{eq: wwntb}
\abs{\sum_{P:P_1\subset P\subset P_2}(-1)^{\dim\aaa_P}
\sum_{\nu\in N_1^P(\Q):I_P(\eta\nu)\subset V}\int_{N_P(\A)}f(a^{-1}\nu na)\ dn}\ll_{K_1}\\
e^{-\sprod{\asmofrts 12}{\Ht_0(a)}}
\sum_{i=1}^m\int_{N_1(\A)}\abs{(f*X_i)(a^{-1}na)}\ dn
\end{multline}
for any $f\in\funct(N_1(\A);K_1)$ and $a\in A_0$ such that $\tau_0^2(\Ht_0(a)-T_1)=1$.
Moreover, if $K_1$ contains $\fund_{P_1,\fin}(\prmlvl)$, then we may take the implied constant to be $\prmlvl^s$, where $s$ is a positive integer depending only on $G$.
\end{proposition}

\begin{proof}
We first remark that the condition $I_P(\eta\nu)\subset V$ is equivalent to $\geom I_P(\eta\nu)\subset V$,
since for any $\gamma\in G(\Q)$ the conjugacy class of $\gamma$ in $G(\Q)$ is Zariski dense in the conjugacy class
of $\gamma$ in $G$.
Also, the condition $\geom I_P(\eta\nu)\subset V$ is equivalent to $\geom I_P^{P_2}(\eta\nu)\subset V'$, where
\[
V'=\{\delta\in M_2:\geom I_{P_2}(\delta)\subset V\},
\]
which is a closed subvariety of $M_2$ by \eqref{eq: monotone}.
We may therefore, after replacing
$f$ by $\int_{N_2(\A)}f(n\cdot)\ dn$, assume without loss of generality that $P_2=G$.

For any $\nu\in N_1(\Q)$ let
\[
\pars(\nu)=\{P\supset P_1 :  \, I_P(\eta\nu)\subset V\}=
\{P\supset P_1 :  \, \geom I_P(\eta\nu)\subset V\}.
\]
Note that by \eqref{eq: monind}, $\pars(\nu)$ is monotone.
Recall that in Lemma \ref{lem: fundom} we have constructed
for any standard parabolic subgroup $P$ of $G$ a family of fundamental domains $\fund_P(\prmlvl)$ for $N_P (\Q) \backslash N_P (\A)$.
Let $\prmlvl$ be a positive integer such that $\fund_{P_1,\fin}(\prmlvl)$
is a subgroup of $K_1$.
Let $\pars$ be a monotone family of parabolic subgroups of $G$ which contain $P_1$ and consider
\[
A(\pars)=\sum_{\nu\in N_1(\Q):\pars(\nu)=\pars}\sum_{P\in\pars}(-1)^{\dim\aaa_P}
\int_{\fund_P(\prmlvl)}f(a^{-1}\nu na)\ dn.
\]
Then the left-hand side of \eqref{eq: wwntb} is equal to $\abs{\sum A(\pars)}$,
where $\pars$ ranges over all monotone families.
Thus, in order to prove the proposition it suffices to show that there exist suitable $X_i$ and $s$ such that
\begin{equation} \label{eq: eachpars}
\abs{A(\pars)}\le {\prmlvl}^s e^{-\sprod{\alpha}{\Ht_0(a)}}\sum_{i=1}^m\int_{N_1(\A)}\abs{(f*X_i)(a^{-1}na)}\ dn
\end{equation}
for any $\alpha\in\Delta_0^2\setminus\Delta_0^1$.
Recall the notation $Q(\pars)$ for the parabolic subgroup generated by the inclusion-minimal elements of $\pars$.
To show \eqref{eq: eachpars}, we distinguish two cases for $\alpha$.
If $\alpha\notin\Delta_0^{Q(\pars)}$ then \eqref{eq: eachpars} follows from Corollary \ref{cor: altsumunipa2}.
For $\alpha\in\Delta_0^{Q(\pars)}$, we will show in fact that
\begin{equation} \label{eq: indvP2}
\abs{\sum_{\nu\in N_1(\Q):\pars(\nu)=\pars}f(a^{-1}\nu a)}\le
\prmlvl^{s'}
e^{-\sprod\alpha{\Ht_0(a)}}\sum_{i=1}^m\int_{N_1(\A)}\abs{(f*X_i)(a^{-1}na)}\ dn
\end{equation}
for all $P\supset P_1$. From this we obtain (with possibly different differential operators $X_i$) the estimate
\begin{equation} \label{eq: indvP}
\abs{\sum_{\nu\in N_1(\Q):\pars(\nu)=\pars}\int_{\fund_P(\prmlvl)}f(a^{-1}\nu na)\ dn}\le {\prmlvl}^s
e^{-\sprod\alpha{\Ht_0(a)}}\sum_{i=1}^m\int_{N_1(\A)}\abs{(f*X_i)(a^{-1}na)}\ dn,
\end{equation}
which in turn implies \eqref{eq: eachpars} for such $\alpha$.
To derive \eqref{eq: indvP} from \eqref{eq: indvP2}, we apply the latter for each $n\in\fund_P(\prmlvl)$
to the right translate $\tilde f$ of $f$ by $a^{-1}na$.
The sets $a^{-1}\fund_P(\prmlvl)a$ for all possible $a$ are contained in a compact subset of $N_P(\A)$
of the form $\Omega_\infty \fund_{P,\fin}(\prmlvl)$, where
the coordinates of the elements of $\log(\Omega_\infty)$ are bounded linearly in $\prmlvl$. Therefore
$\tilde f\in\funct(N_1(\A); \fund_{P_1,\fin}(\prmlvl))$, and the coordinates of $\Ad(a^{-1}na)X_i$ with respect to a fixed basis of $\univ(\nnn)$ are bounded
polynomially in $\prmlvl$.
Thus, \eqref{eq: indvP} follows from \eqref{eq: indvP2}.

To show \eqref{eq: indvP2}, let $R\in\pars_{\min}$ be such that $\alpha\in\Delta_0^R$
and let $S$ be the standard parabolic subgroup of $G$ such that $\Delta_0^S=\Delta_0^R\setminus\{\alpha\}$.
Thus, $P_1\subset S\subsetneq R$.
We claim that the left-hand side of \eqref{eq: indvP2} is majorized by
\begin{equation} \label{eq: nu123}
\sum_{\nu_1\in N_1^S(\Q)}\sum_{\nu_2\in N_S^R(\Q):\eta\nu_1\nu_2\notin\geom I_S^R(\eta\nu_1)}
\sum_{\nu_3\in N_R(\Q)}\abs{f(a^{-1}\nu_1\nu_2\nu_3 a)}.
\end{equation}
Indeed, if $\pars(\nu_1\nu_2\nu_3)=\pars$ then $\eta\nu_1\nu_2\notin\geom I_S^R(\eta\nu_1)$,
for otherwise, by \eqref{eq: transitive}
we would have
\[
\geom I_S(\eta\nu_1\nu_2\nu_3)=\geom I_S(\eta\nu_1)=\geom I_R(\eta\nu_1\nu_2)=\geom I_R(\eta\nu_1\nu_2\nu_3),
\]
in contradiction to the minimality of $R$ (cf.~\cite[Lemme 6.7.5]{MR3427596}).

Using Lemma \ref{lem: latticesum}, we bound the inner sum in \eqref{eq: nu123} by
\[
\prmlvl^{s_1} \sum_i\int_{N_R(\A)}\abs{(f*X_i')(a^{-1}\nu_1\nu_2n_3 a)}\ dn_3
\]
for suitable $X_1',\dots,X'_{m'}\in\univ(\nnn_R)$.
By Corollary \ref{cor: savephi} and Lemma \ref{lem: bndeg} we can now bound the sum over $\nu_2$ in \eqref{eq: nu123} by
\[
\prmlvl^{s_2} e^{-\sprod{\alpha}{\Ht_0(a)}}\sum_i\int_{N_S(\A)}\abs{(f*X_i'')(a^{-1}\nu_1n_2 a)}\ dn_2
\]
for suitable $X_1'',\dots,X_{m''}''\in\univ(\nnn_S)$.
Using Lemma \ref{lem: latticesum} once again, we bound \eqref{eq: nu123} by the right-hand side of \eqref{eq: indvP2},
as required.
\end{proof}

The following corollary follows from this proposition exactly like Corollary \ref{cor: altsumunip} follows from Proposition
\ref{prop: altsumunipa}.

\begin{corollary} \label{cor: subsetY}
Let $P_1\subset P_3\subset P_2$ be standard parabolic subgroups of $G$.
There exist an integer $s$ and $X_1,\dots,X_m\in\univ(\Lieg^1)$ such that for any closed subvariety $V$ of $G$ and $\eta\in M_3(\Q)$ we have
\begin{multline}
\abs{\sum_{P:P_3\subset P\subset P_2}(-1)^{\dim\aaa_P}
\sum_{\nu\in N_3^P(\Q):I_P(\eta\nu)\subset V}\int_{N_P(\A)}f(g^{-1}\nu ng)\ dn}\le \\ \level (K)^s
e^{-\sprod{\asmofrts 32}T}e^{-\sprod{(\asmofrts 32)_{P_1}}{\Ht_0(g)-T}}
\sum_{i=1}^m\int_{N_3(\A)}\abs{(f*X_i)(g^{-1}ng)}\ dn
\end{multline}
for any $f\in\funct(G(\A)^1;K)$ and $g\in G(\A)^1$ such that $F^1(g,T)\tau_1^2(H_{P_1}(g)-T)=1$.
\end{corollary}

Recall the equivalence relation $\sim$
on $G(\Q)$ introduced in the introduction:
for $\gamma,\delta\in G(\Q)$ we write $\gamma\sim\delta$ if $\gamma\sim_w\delta$ and $\gamma$ and $\delta$ are conjugate in
$G(\overline{\Q})$. Let $\Orb$ be the set of equivalence classes of $\sim$.

Let $\orb=\corb\cap\conj$ be an equivalence class of $\sim$, where $\corb$ is a coarse class and $\conj$ is a geometric conjugacy class of $G$, and let $\overline{\conj}$ be the Zariski
closure of $\conj$.
Applying Corollary \ref{cor: subsetY} to the closed varieties $V=\overline{\conj}$ and $V=\overline{\conj}\setminus\conj$ and subtracting, we obtain:
\begin{corollary} \label{cor: altsumorb}
Let $P_1\subset P_3\subset P_2$ be standard parabolic subgroups of $G$.
There exist an integer $s$ and $X_1,\dots,X_m\in\univ(\Lieg^1)$ such that for any $\eta\in M_3(\Q)$ and $\orb\in\Orb$ we have
\begin{multline}
\abs{\sum_{P:P_3\subset P\subset P_2}(-1)^{\dim\aaa_P}
\sum_{\nu\in N_{3}^P(\Q):I_P(\eta\nu)=\orb}\int_{N_P(\A)}f(g^{-1}\nu ng)\ dn}\le \\ \level (K)^s
e^{-\sprod{\asmofrts 32}T}e^{-\sprod{(\asmofrts 32)_{P_1}}{\Ht_0(g)-T}}
\sum_{i=1}^m\int_{N_3(\A)}\abs{(f*X_i)(g^{-1}ng)}\ dn
\end{multline}
for any $f\in\funct(G(\A)^1;K)$ and $g\in G(\A)^1$ such that $F^1(g,T)\tau_1^2(H_{P_1}(g)-T)=1$.
\end{corollary}

\section{The main result} \label{sec: main}
We are now ready to prove the continuity of the decomposition
of the geometric side with respect to the equivalence relation
$\sim$.
Observe that if $\sigma\in G(\Q)$ is semisimple then the equivalence classes with respect to $\sim$ in the coarse class of $\sigma$
are indexed by the geometric unipotent conjugacy classes of the centralizer of $\sigma$ containing a rational point.
In particular,
\begin{equation} \label{eq: fincoarse}
\text{the number of equivalence classes of $\sim$ in a coarse class is bounded in terms of $G$ only.}
\end{equation}


Following \cite{MR3427596, 1412.8673}, we define for any $\orb\in\Orb$
\[
k_{\orb}(x)=\sum_{\gamma\in\orb}f(x^{-1}\gamma x)
\]
and
\[
k^T_{\orb}(x)=\sum_{P\supset P_0}(-1)^{\dim\aaa_P}\sum_{\delta\in P(\Q)\bs G(\Q)}k_{\orb,P}(\delta x)\hat\tau_P(H_P(\delta x)-T),
\]
where
\[
k_{\orb,P}(x)=\sum_{\gamma\in M_P(\Q):I_P(\gamma)=\orb}\int_{N_P(\A)}f(x^{-1}\gamma nx)\ dn.
\]
Let
\[
J^T_\orb(f)=\int_{G(\Q)\bs G(\A)^1}k^T_\orb(x)\ dx.
\]

Exactly as before, we can write
\[
k^T_{\orb}(x)=\sum_{P_1\subset P_2}\sum_{\delta\in P_1(\Q)\bs G(\Q)}F^{P_1}(\delta x,T)\sigma_1^2(H_{P_1}(\delta x)-T)k_{\orb,1,2}(\delta x),
\]
where
\[
k_{\orb,1,2}(x)=\sum_{P:P_1\subset P\subset P_2}(-1)^{\dim\aaa_P}k_{\orb,P}(x).
\]
Clearly,
\[
k^T(x)=\sum_{\orb\in\Orb}k^T_{\orb}(x).
\]

\begin{theorem} \label{thm: main}
The analogue of Theorem \ref{thm: maincoarse} holds
for the integrals $J^T_\orb (f)$, $\orb \in \Orb$, instead of
$J^T_\corb (f)$, $\corb \in \cOrb$.
\end{theorem}

\begin{proof}
The proof is very similar to that of Theorem \ref{thm: maincoarse}.
The main task is to prove the analogue of part \ref{part: estimcoarse} for the integrals $J^T_\orb (f)$.
We write as before
\[
k_{\orb,1,2}(x)=\sum_{P_1\subset P_3\subset P_2}k_{\orb,1,2;3}(x),
\]
where $k_{\orb,1,2;3}$ is the left $P_1(\Q)N_2(\A)$-invariant function given by
\[
k_{\orb,1,2;3}(x)=\sum_{\eta\in M_3(\Q)\rg{1}}\sum_{P:P_3\subset P\subset P_2}(-1)^{\dim\aaa_P}
\sum_{\nu\in N_3^P(\Q):I_P(\eta\nu)=\orb}\int_{N_P(\A)}f(x^{-1}\eta\nu nx)\ dn.
\]
We show that there exist $r\ge0$ and a continuous seminorm $\mu$ on $\funct(G(\A)^1;K)$, such that for any triplet
$P_1\subset P_3\subset P_2$ with $P_1\neq P_2$ we have
\[
\sum_{\orb}\int_{P_1(\Q)\bs G(\A)^1}F^1(x,T)\sigma_1^2(H_{P_1}(x)-T)\abs{k_{\orb,1,2;3}(x)}\ dx
\le \mu(f)(1+\norm{T})^re^{-d(T)}.
\]
For that we use Corollary \ref{cor: altsumorb} (applied to suitable left translates of $f$) to obtain
\[
\abs{k_{\orb,1,2;3}(x)}\le
e^{-\sprod{\asmofrts 32}T}e^{-\sprod{(\asmofrts 32)_{P_1}}{\Ht_0(x)-T}}
\sum_i\sum_{\eta\in M_3(\Q)\rg{1}\cap\corb}\int_{N_3(\A)}\abs{(f*X_i) (x^{-1}\eta nx)}\ dn,
\]
provided that $F^1(x,T)\sigma_1^2(H_{P_1}(x)-T)=1$, where $\corb$ is the coarse class containing $\orb$.
Taking into account Remark \ref{rem: xitxi} and \eqref{eq: fincoarse},
the rest of the argument proceeds as in the proof of Theorem \ref{thm: maincoarse}.
\end{proof}

We continue to write $J^T_{\orb}(f)$ for the value at $T$ of the polynomial $J^T_{\orb}(f)$, even if $d(T)\le d_0$.

\begin{corollary} \label{CorollaryContinuity}
For any $T\in\aaa_0$,
\[
f\mapsto\sum_{\orb\in\Orb}\abs{J^T_{\orb}(f)}
\]
is a continuous seminorm on $\funct(G(\A)^1)$.
On $\funct(G(\A)^1;K)$ this seminorm is bounded by
$c(T) \level (K)^s \norm{\cdot}_{G(\A)^1,\, t}$ for
a constant $c(T)$ and positive integers $s, t$ that do not depend on $K$.
\end{corollary}

\begin{proof}
By extrapolation, it is enough to prove this for $d(T)>d_0$.
The claim follows immediately from Theorem \ref{thm: main} and the fact (see \eqref{eq: P=G}) that
\[
\sum_{\corb\in\cOrb}\int_{G(\Q)\bs G(\A)^1_{\le T}}\abs{k_{\corb}(x)}\ dx
\]
is a continuous seminorm on $\funct(G(\A)^1;K)$.
\end{proof}

\begin{remark} \label{FinerRelation}
Motivated by \cite[\S8]{MR835041}, we may consider instead of the equivalence relation $\sim$ on $G(\Q)$, the slightly finer relation
$\gamma\fsim\delta$ if $\delta=g\gamma g^{-1}$ for some $g\in G(\Q)G_{\spart{\gamma}}(\overline{\Q})$.
Thus, the $\fsim$-classes in the coarse class of a semisimple element $\sigma\in G(\Q)$ are indexed by the geometric unipotent classes
of $G_\sigma$ containing a rational point.
In particular, in the case where $G^{\der}$ is simply connected, $\fsim$ coincides with $\sim$, since the centralizers of semisimple elements are then connected (cf. \cite[pp. 788--789]{MR683003}).
In general, Theorem \ref{thm: main} continues to hold for the classes of $\fsim$.
The proof is identical except that in the case at hand $I^G_P (\gamma)$ will be defined to be the $\fsim$-class
of an element of $\spart{\gamma}I_{P_{\spart{\gamma}}}^{G_{\spart{\gamma}}}(\upart{\gamma})$.
In practice, this refinement is not very essential, since in most applications of the trace formula we can reduce to the case where $G^{\der}$ is simply connected
by considering a $z$-extension of $G$.
\end{remark}

\def\cprime{$'$}
\providecommand{\bysame}{\leavevmode\hbox to3em{\hrulefill}\thinspace}
\providecommand{\MR}{\relax\ifhmode\unskip\space\fi MR }
\providecommand{\MRhref}[2]{%
  \href{http://www.ams.org/mathscinet-getitem?mr=#1}{#2}
}
\providecommand{\href}[2]{#2}

\end{document}